\documentclass[12pt,a4paper]{amsart}

\usepackage{amssymb}
\usepackage{amsthm}
\usepackage{amsmath}
\usepackage{graphicx}
\usepackage{enumerate}
\usepackage{xcolor}
\usepackage{tikz}
\usepackage{graphbox}
\usepackage{orcidlink}
\usepackage{hyperref}
\hypersetup{colorlinks=true,allcolors=blue}
\usepackage[margin=2.5cm]{geometry}

\newtheorem{theorem}{Theorem}

\newtheorem{corollary}[theorem]{Corollary}
\newtheorem{lemma}[theorem]{Lemma}

\theoremstyle{definition}
\newtheorem{example}[theorem]{Example}

\newtheorem{remark}[theorem]{Remark}

\newcommand{\II}{{\mathbb I}}
\newcommand{\RR}{\mathbb{R}}

\newcommand{\BB}{\mathcal{B}}

\newcommand{\KK}{\mathcal{K}}
\newcommand{\LL}{\mathcal{L}}
\newcommand{\muhat}{\widehat{\mu}}
\newcommand{\muover}{\overline{\mu}}

\title{Bivariate measure-inducing quasi-copulas}

\author[N. Stopar]{Nik Stopar \orcidlink{0000-0002-0004-4957} } 
\address{University of Ljubljana, Faculty of Civil and Geodetic Engineering, and Institute of Mathematics, Physics and Mechanics, Ljubljana, Slovenia}
\email{nik.stopar@fgg.uni-lj.si}


\begin{document}

\begin{abstract}
It is well known that every bivariate copula induces a positive measure on the Borel $\sigma$-algebra on $[0,1]^2$, but there exist bivariate quasi-copulas that do not induce a signed measure on the same $\sigma$-algebra. In this paper we show that a signed measure induced by a bivariate quasi-copula can always be expressed as an infinite combination of measures induced by copulas. With this we are able to give the first characterization of measure-inducing quasi-copulas in the bivariate setting.
\end{abstract}

\maketitle

\section{Introduction}

Copulas, introduced in 1959 by Sklar, are one of the main tools for modeling dependence of random variables in statistical literature. These are multivariate functions that link the cumulative distribution function of a random vector and its one-dimensional marginal distributions. Copulas have found widespread use in various practical applications such as finance \cite{McNFreEmb15}, biology \cite{OnkGruMunObe09}, environmental sciences \cite{BulYie20,DzuNgaOdo20} and many others.

Quasi-copulas, a generalization of copulas, were introduced by Alsina, Nelsen, and Schweizer \cite{AlsNelSch93} in order to characterize operations on distribution functions that can be derived from operations on random variables. Their importance is explained by the following property: a point-wise infimum respectively supremum of a given set of copulas is always a quasi-copula.
Thus, quasi-copulas are indispensable in the theory of imprecise probabilities, which model situations when the exact dependence structure between random variables, i.e. copula, is not known and is therefore replaced by a set of copulas.

The set of all quasi-copulas has been studied intensively in recent years and compared to the set of all copulas. The lattice theoretical properties of both sets were investigated in \cite{NelUbe05,FerNelUbe11,AriDeB19,OmlSto22-2}, while topological properties, particularly from the perspective of Baire categories, were studied in \cite{DurFerTru16,DurFerTruUbe20}.
From a measure-theoretic point of view one of the major differences between copulas and quasi-copulas is that, while every $n$-variate copula $C$ induces a positive measure $\mu_C$ on the Borel $\sigma$-algebra on $[0,1]^n$ (see \cite{Nel06}), there exist $n$-variate quasi-copulas (for all $n \ge 2$) that do not induce a signed measure on the same $\sigma$-algebra \cite{NelQueRodUbe10,FerRodUbe11}.
This has stimulated numerous investigations of the mass distribution of quasi-copulas \cite{NelQueRodUbe02,DeBDeMUbe07,Ube23,Sem23,DolKuzSto24} and related concepts \cite{DibSamMesKle16,OmlSto20,OmlSto22}, aimed at a better understanding of the behaviour of quasi-copulas.
As evidenced by several very recent papers, this is an active area of research, where there is still much to be done. In fact, a full characterization of quasi-copulas that do induce a signed measure on the Borel $\sigma$-algebra on $[0,1]^n$ is still an open problem, see \cite[Problem~4]{AriMesDeB20}.

In this paper we study bivariate quasi-copulas that induce a signed measure on the Borel $\sigma$-algebra on $[0,1]^2$. 
We show that the signed measure $\mu_Q$ induced by such a quasi-copula $Q$ can always be expressed with measures induced by bivariate copulas.
This is done by making use of a recent result in \cite{DolKuzSto24} giving a characterization of those quasi-copulas that can be expressed as linear combinations of two copulas. Note that any such quasi-copula automatically induces a signed measure, but not all measure-inducing quasi-copulas can be expressed as linear combinations of copulas, see \cite[Example~13]{DolKuzSto24}.
However, the same paper also initiated the study of quasi-copulas using infinite series of copulas and this idea is key to our result.
In particular, our main theorem reads as follows.

\begin{theorem}\label{thm:main}
   For a bivariate quasi-copula $Q$ the following two conditions are equivalent.
    \begin{enumerate}[$(i)$]
        \item There exists a signed measure $\mu_Q$ defined on the Borel $\sigma$-algebra on $[0,1]^2$ such that
        \begin{equation*}
            \mu_Q([0,x] \times [0,y])=Q(x,y) \qquad\text{for all $x,y \in [0,1]$.}
        \end{equation*}
        \item There exists a sequence of bivariate copulas $C_n$ and a sequence of real numbers $\gamma_n$ such that
        \begin{enumerate}[$(a)$]
            \item the series of functions $\sum_{n=1}^\infty \gamma_n C_n$ converges uniformly to $Q$ and
            \item the series of induced measures $\sum_{n=1}^\infty \gamma_n \mu_{C_n}$ converges in the total variation norm to some finite signed measure.
        \end{enumerate}
    \end{enumerate}
\end{theorem}

This result can be seen as a characterization of bivariate quasi-copulas that induce a signed measure on $[0,1]^2$, so it gives an answer to the open problem \cite[Problem~4]{AriMesDeB20} mentioned above in the bivariate case. However, we do not consider the bivriate case to be completely resolved, since it would still be beneficial to obtain a characterization of measure-inducing quasi-copulas in operationally simpler terms.

The paper is structured as follows. In Section~\ref{sec:pre} we recall some basic notions from measure theory and some results on bivariate quasi-copulas that we will need in our proofs.
The main part of the paper is devoted to the proof of Theorem~\ref{thm:main}.
Assuming a quasi-copula $Q$ induces a signed measure $\mu_Q$, we construct in Section~\ref{sec:measure} a sequence of measure-inducing quasi-copulas $Q_N$ that converge to $Q$ and whose induced measures converge to $\mu_Q$. In addition, all these quasi-copulas are linear combinations of copulas.
In Section~\ref{sec:proof} we convert the sequence $Q_N$ into a series of multiples of copulas $C_n$, and finally prove Theorem~\ref{thm:main}.
An example that demonstrates our result is given in Section~\ref{sec:example}.

\section{Preliminaries}\label{sec:pre}

Throughout the paper we will denote the unit interval by $\II=[0,1]$ and the unit square by $\II^2$. A function $Q \colon \II^2 \to \II$ is a (bivariate) \emph{quasi-copula} if it satisfies the following conditions:
\begin{enumerate}[$(i)$]
    \item $Q$ is \emph{grounded}: $Q(x,0)=Q(0,y)=0$ for all $x,y \in \II$,
    \item $Q$ has \emph{uniform marginals}: $Q(x,1)=x$ and $Q(1,y)=y$ for all $x,y \in \II$,
    \item $Q$ is \emph{increasing} in each variable,
    \item $Q$ is \emph{$1$-Lipschitz}:
    $$|Q(x_2,y_2)-Q(x_1,y_1)| \le |x_2-x_1|+|y_2-y_1|$$
    for all $x_1,y_1,x_2,y_2 \in \II$.
\end{enumerate}
A function $Q \colon \II^2 \to \II$ that satisfies conditions $(i)$, $(ii)$, and
\begin{enumerate}[$(i)$]
\setcounter{enumi}{4}
    \item $Q$ is \emph{$2$-increasing}:
    $$Q(x_2,y_2)-Q(x_1,y_2)-Q(x_2,y_1)+Q(x_1,y_1) \ge 0$$
    for all $x_1,y_1,x_2,y_2 \in \II$ with $x_1 \le x_2$ and $y_1 \le y_2$,
\end{enumerate}
is called a (bivariate) \emph{copula}.

Let $(X,\mathcal{A})$ be a measurable space equipped with a finite signed measure $\mu$.
Let $\mu=\mu^+-\mu^-$ be the Jordan decomposition of measure $\mu$, i.e., $\mu^+$ and $\mu^-$ are finite positive measures with disjoint supports.
If $S^+$ and $S^-$ denote the supports of $\mu^+$ and $\mu^-$, respectively, then $\mu^+(A)=\mu(A \cap S^+)$ and $\mu^-(A)=-\mu(A \cap S^-)$ for all $A \in \mathcal{A}$.
The positive measure $|\mu|=\mu^+ + \mu^-$ is called the \emph{total variation measure} of $\mu$. It satisfies the inequality $|\mu(A)| \le |\mu|(A)$ for all $A \in \mathcal{A}$. The \emph{total variation norm} of $\mu$ is given by
\begin{equation}\label{eq:TV}
    \|\mu\|_{TV}=|\mu|(X)=\mu^+(X)+\mu^-(X) =\mu(S^+)-\mu(S^-).
\end{equation}
The vector space of all finite signed measures on $(X,\mathcal{A})$ equipped with the total variation norm is a Banach space.
For two finite measures $\mu_1$ and $\mu_2$ on $\mathcal{A}$ we write $\mu_1 \le \mu_2$ if $\mu_1(A) \le \mu_2(A)$ for all $A \in \mathcal{A}$.
In particular, we have $-\mu^- \le \mu \le \mu^+$ for any finite signed measure $\mu$.

The Borel $\sigma$-algebras on $\II$ and $\II^2$ will be denoted by $\BB(\II)$ and $\BB(\II^2)$, respectively. Note that $\BB(\II^2) =\BB(\II) \otimes \BB(\II)$ is the smallest $\sigma$-algebra on $\II^2$ that contains all rectangles of the form $R=[0,x]\times [0,y]$ for some $x,y \in \II$. It also contains all rectangles that are either closed of open on any of their four sides. In particular, it contains all half-open rectangles of the form $R=(x_1,x_2]\times (y_1,y_2]$ for some $x_1,x_2,y_1,y_2 \in \II$ with $x_1<x_2$ and $y_1<y_2$.

A bivariate quasi-copula $Q$ is said to \emph{induce a signed measure} on $\BB(\II^2)$ if there exist a signed measure $\mu_Q$ on $\BB(\II^2)$ such that $\mu_Q([0,x] \times[0,y])=Q(x,y)$ for all $(x,y) \in \II^2$. Measure $\mu_Q$ is automatically finite and for any half-open rectangle $R=(x_1,x_2]\times (y_1,y_2]  \subseteq \II^2$ we have $\mu_Q(R)=V_Q(R)$, where
$$V_Q(R)=Q(x_2,y_2)-Q(x_1,y_2)-Q(x_2,y_1)+Q(x_1,y_1)$$
is the \emph{volume} of $R$ with respect to $Q$.
Since quasi-copulas are continuous functions, the equality $\mu_Q(R)=V_Q(R)$ holds also if the rectangle $R$ is either open or closed on any of its sides, because the $\mu_Q$ measure of a vertical or horizontal segment is $0$. For example,
\begin{align*}
    &\mu_Q\big((x_1,x_2]\times \{y_0\}\big)\\
    &=\mu_Q\Big(\bigcap_{n=1}^\infty (x_1,x_2]\times (\tfrac{n-1}{n}y_0,y_0]\Big)
    =\lim_{n \to \infty} \mu_Q\big((x_1,x_2]\times (\tfrac{n-1}{n}y_0,y_0]\big)\\
    &=\lim_{n \to \infty} \big(Q(x_2,y_0)-Q(x_1,y_0)-Q(x_2,\tfrac{n-1}{n}y_0)+Q(x_1,\tfrac{n-1}{n}y_0) \big)\\
    &=Q(x_2,y_0)-Q(x_1,y_0)-Q(x_2,y_0)+Q(x_1,y_0)=0.
\end{align*}
It is well known that any copula $C$ induces a (positive) measure $\mu_C$ on $\BB(\II^2)$ and this measure is \emph{stochastic} in the sense that $\mu_C(A \times \II)=\mu_C(\II \times A)=\lambda(A)$, where $\lambda$ denotes the Lebesgue measure on $\II$.

We recall a result from \cite{DolKuzSto24} that will be crucial for our constructions. For a positive integer $m$ we will denote $[m]=\{1,2,\ldots,m\}$.

\begin{theorem}[{\cite[Theorem~10]{DolKuzSto24}}]\label{thm:LCC}
    For a bivariate quasi-copula $Q$ the following conditions are equivalent.
    \begin{enumerate}[$(i)$]
        \item There exist bivariate copulas $A$ and $B$ and real numbers $\alpha$ and $\beta$ such that
        \begin{equation*}
            Q(x,y)=\alpha A(x,y) + \beta B(x,y) \qquad\text{for all $(x,y) \in \II^2$.}
        \end{equation*}
    \item Quasi-copula $Q$ satisfies the condition $\alpha_Q <\infty$, where
    \begin{equation*}
        \alpha_Q=\displaystyle\sup_{n \geq 1} \left\{ \max_{k \in [2^n]} 2^n \sum_{l=1}^{2^n} V_Q(R_{kl}^{n})^+ \, , \, \max_{l \in [2^n]} 2^n \sum_{k=1}^{2^n} V_Q(R_{kl}^{n})^+ \right\},
    \end{equation*}
    $R_{kl}^{n}=[\tfrac{k-1}{2^n},\tfrac{k}{2^n}] \times [\tfrac{l-1}{2^n},\tfrac{l}{2^n}]$ for all $k,l \in [2^n]$, and $x^+=\max\{x,0\}$ for all $x \in \RR$.
    \end{enumerate}
\end{theorem}

Informally speaking, what the coefficient $\alpha_Q$ does, is approximately detect at most how much positive mass quasi-copula $Q$ accumulates on any vertical/horizontal strip relative to the width/height of the strip. With increasing $n$ the detection is more and more accurate.

\section{Measure induced by a quasi-copula}\label{sec:measure}

We assume throughout this section that $Q$ is a bivariate quasi-copula that induces a signed measure $\mu_Q$ on $\BB(\II^2)$. We will denote by $\mu_Q=\mu_Q^+ -\mu_Q^-$ the Jordan decomposition of measure $\mu_Q$.
Since $\mu_Q^+(\II^2)-\mu_Q^-(\II^2)=1$, both $\mu_Q^+$ and $\mu_Q^-$ are finite measures. Measure $\mu_Q$ is stochastic, but measures $\mu_Q^+$ and $\mu_Q^-$ are not, unless $Q$ is a copula, in which case $\mu_Q^+=\mu_Q$ is stochastic and $\mu_Q^-$ is the zero measure. 

The goal of this section is to construct a sequence of bivariate quasi-copulas $Q_N$ with the following properties:
\begin{enumerate}[$(a)$]
    \item sequence $Q_N$ converges to $Q$ uniformly,
    \item for every $N$, $Q_N$ induces a signed measure $\mu_N$,
    \item for every $N$, $Q_N$ is a linear combination of two copulas,
    \item sequence $\mu_N$ converges to $\mu_Q$ in the total variation norm.
\end{enumerate}
%We remark that $(b)$ is actually a consequence of $(c)$, and $(a)$ is a consequence of $(b)$ and $(d)$, however, our proofs will follow a different order and property $(c)$ will actually be proved last.

Let us give an outline of the construction, which is split into several steps to make it easier to follow.
The main idea is to start with condition $(c)$, using Theorem~\ref{thm:LCC}.
We need to approximate quasi-copula $Q$, which need not satisfy $\alpha_Q<\infty$ (see condition~$(ii)$ of Theorem~\ref{thm:LCC}), with a quasi-copula $Q_N$, that does satisfy $\alpha_{Q_N}<\infty$.
To this end we identify sets of the form $K_N \times \II$ and $\II \times L_N$ which cause $\alpha_Q$ to be greater than $cN$ for some fixed normalising constant $c$. These sets are defined in Subsection~\ref{subs:sets}.
Quasi-copula $Q_N$ (for every $N \ge 1$) is constructed in Subsection~\ref{subs:quasi-copulas} by "smoothing out" the mass distribution of $Q$ on the set $(K_N \times \II) \cup (\II \times L_N)$ and leaving it unchanged elsewhere.
For this to work, the sets $K_N$ and $L_N$ need to be constructed carefully, because, to ensure property~$(a)$, they need to have small Lebesgue measure, i.e., in the limit when $N$ goes to infinity the measure must tend to $0$, and, to ensure property~$(d)$, the sets $(K_N \times \II) \cup (\II \times L_N)$ need to have small $|\mu_Q|$ measure in the same sense.
Both of these properties are verified in Subsection~\ref{subs:properties}.
We then prove property~$(b)$ by explicitly constructing signed measures $\mu_N$ and showing with a direct calculation that they are induced by quasi-copulas $Q_N$.
This is done in Subsection~\ref{subs:measures}, where property~$(d)$ is also verified.
Finally, in Subsection~\ref{subs:decomposition} we prove that quasi-copulas $Q_N$ satisfy condition~$(ii)$ of Theorem~\ref{thm:LCC}, which then give us property~$(c)$ and consequently also property $(a)$.

\subsection{Construction of sets \texorpdfstring{$K_N$}{K\_N}}\label{subs:sets}\phantom{}\medskip

We start with the construction of sets $K_N$, the sets $L_N$ will be defined later.
For an integer $n \ge 0$ we introduce a family of open intervals
\begin{equation*}
    \mathcal{P}_n=\big\{(\tfrac{i-1}{2^n},\tfrac{i}{2^n}) \mid i \in [2^n]\big\},
\end{equation*}
which essentially form a partition of $\II$.
For integers $n \ge 0$ and $N \ge 1$ let
\begin{equation}\label{eq:J_family}
    \mathcal{J}_{n,N}=\Big\{S \in \mathcal{P}_n \mid 2^n \mu_Q^+(S \times \II) > N \cdot \mu_Q^+(\II^2)\Big\} \subseteq \mathcal{P}_n \qquad\text{and}\qquad J_{n,N}=\bigcup_{S \in \mathcal{J}_{n,N}} S.
\end{equation}
Note that for every $S \in \mathcal{P}_n$ we have $\lambda(S)=\frac{1}{2^n}$, so that
\begin{equation}\label{eq:J_cond}
    S \in \mathcal{J}_{n,N} \qquad\text{if and only if}\qquad S \in \mathcal{P}_n \text{ and }\frac{\mu_Q^+(S \times \II)}{\lambda(S)} > N \cdot \mu_Q^+(\II^2).
\end{equation}
To make sense of what follows, we make the following remark.
Intuitively, the sets $J_{n,N} \times \II$ are the "bad" strips, which make $\alpha_Q$ from condition $(ii)$ of Theorem~\ref{thm:LCC} large and possibly infinite.
We will later "smooth out" the mass distribution of $Q$ on the bad strips in order to make $\alpha_Q$ finite. For this to work, we actually need to slightly enlarge the sets $J_{n,N}$, so that the smoothing will also affect the boundary of the bad strips.

For every $n \ge 1$ and $S \in \mathcal{P}_n$ there exists a uniquely determined $\widehat{S} \in \mathcal{P}_{n-1}$ such that $S \subseteq \widehat{S}$. For all $N \ge 1$ let
\begin{equation}\label{eq:K_family_0}
    \KK_{0,N}=\big\{\widehat{S} \mid S \in \mathcal{J}_{1,N}\big\} \subseteq \mathcal{P}_0 \qquad\text{and}\qquad K_{0,N}=\bigcup_{A \in \KK_{0,N}} A,
\end{equation}
and define inductively for all $n \ge 1$
\begin{equation}\label{eq:K_family}
    \KK_{n,N}=\Big\{\widehat{S} \mid S \in \mathcal{J}_{n+1,N},\  S \cap \bigcup_{k=0}^{n-1} K_{k,N} = \emptyset\Big\} \subseteq \mathcal{P}_n \qquad\text{and}\qquad K_{n,N}=\bigcup_{A \in \KK_{n,N}} A.
\end{equation}

We give an example to demonstrate the sets introduced above.

\begin{example}\label{ex:sets_K}
Let $Q$ be a quasi-copula with mass distributed as depicted in Figure~\ref{fig:example_Q} left, where the unit square $\II^2$ is divided into $16 \times 16$ small squares of dimensions $\frac{1}{16} \times \frac{1}{16}$.
The dark gray squares contain a mass of $\frac{1}{16}$ while the light gray squares contain a mass of $-\frac{1}{16}$ distributed uniformly over the square. All other squares contain no mass.

\begin{figure}[ht]
    \centering
    \includegraphics[width=0.43\textwidth]{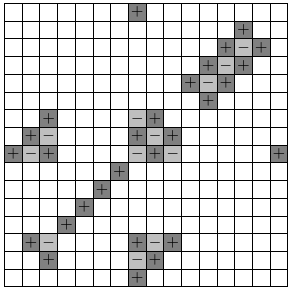} ~~
    \includegraphics[width=0.56\textwidth]{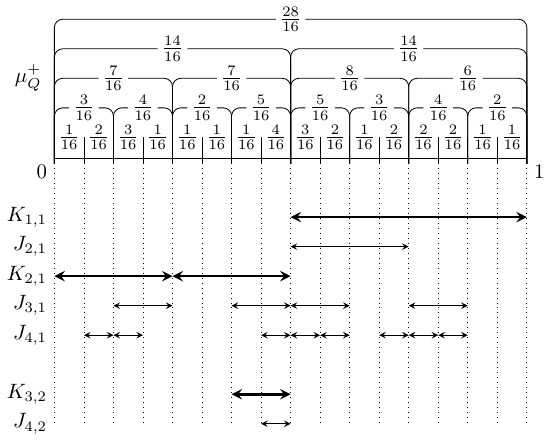}
    \caption{Mass distribution of a quasi-copula $Q$ from Example~\ref{ex:sets_K} (left) and the corresponding sets $J_{n,N}$ and $K_{n,N}$ (right). Each dark gray square contains a mass of $\frac{1}{16}$ and each light gray square contains a mass of $-\frac{1}{16}$.}
    \label{fig:example_Q}
\end{figure}

The corresponding (nonempty) sets $J_{n,N}$ and $K_{n,N}$, where $0\le n\le 4$ and $1\le N \le 2$, are depicted on Figure~\ref{fig:example_Q} right, along with the values $\mu_Q^+(S \times \II)$ for $S \in \mathcal{P}_0 \cup \mathcal{P}_1 \cup \mathcal{P}_2 \cup \mathcal{P}_3 \cup \mathcal{P}_4$ on top (note that $\mu_Q^+(\II^2)=\frac{28}{16}$).
In particular,
\begin{equation*}
J_{0,1}=\emptyset,\quad
J_{1,1}=\emptyset,\quad
J_{2,1}=(\tfrac{2}{4},\tfrac{3}{4}),\quad
J_{3,1}=(\tfrac{1}{8},\tfrac{2}{8})\cup(\tfrac{3}{8},\tfrac{4}{8})\cup(\tfrac{4}{8},\tfrac{5}{8})\cup(\tfrac{6}{8},\tfrac{7}{8}),
\end{equation*}
\begin{equation*}
J_{4,1}=(\tfrac{1}{16},\tfrac{2}{16})\cup(\tfrac{2}{16},\tfrac{3}{16})\cup(\tfrac{7}{16},\tfrac{8}{16})\cup(\tfrac{8}{16},\tfrac{9}{16})\cup(\tfrac{9}{16},\tfrac{10}{16})\cup(\tfrac{11}{16},\tfrac{12}{16})\cup(\tfrac{12}{16},\tfrac{13}{16})\cup(\tfrac{13}{16},\tfrac{14}{16}),
\end{equation*}
\begin{equation*}
K_{0,1}=\emptyset,\quad
K_{1,1}=(\tfrac{1}{2},1),\quad
K_{2,1}=(0,\tfrac{1}{4})\cup(\tfrac{1}{4},\tfrac{2}{4}),\quad
K_{3,1}=\emptyset,\quad
K_{4,1}=\emptyset,
\end{equation*}
\begin{equation*}
J_{0,2}=\emptyset,\quad
J_{1,2}=\emptyset,\quad
J_{2,2}=\emptyset,\quad
J_{3,2}=\emptyset,\quad
J_{4,2}=(\tfrac{7}{16},\tfrac{8}{16}),
\end{equation*}
\begin{equation*}
K_{0,2}=\emptyset,\quad
K_{1,2}=\emptyset,\quad
K_{2,2}=\emptyset,\quad
K_{3,2}=(\tfrac{3}{8},\tfrac{4}{8}),\quad
K_{4,2}=\emptyset.
\end{equation*}

\end{example}

Next, we give some basic properties of the sets defined above.

\begin{lemma}\label{lem:J_n_N}
For all $n \ge 1$ we have
\begin{equation*}
J_{n,N} \subseteq \bigcup_{k=0}^{n-1} K_{k,N}.
\end{equation*}
\end{lemma}

\begin{proof}
    We prove this by induction on $n$. From equations \eqref{eq:J_family} and \eqref{eq:K_family_0} if follows that $J_{1,N}=\bigcup_{S \in \mathcal{J}_{1,N}} S \subseteq \bigcup_{S \in \mathcal{J}_{1,N}} \widehat{S}=K_{0,N}$, so the claim holds for $n=1$.
    Suppose it holds for some $n \ge 1$ and let $S \in \mathcal{J}_{n+1,N}$.
    Then, by equation \eqref{eq:K_family}, either $\widehat{S} \in \KK_{n,N}$ or $S \cap \bigcup_{k=0}^{n-1} K_{k,N} \ne \emptyset$.
    In the first case $S \subseteq \widehat{S} \subseteq K_{n,N}$.
    In the second case $S \cap K_{k,N} \ne \emptyset$ for some $0 \le k \le n-1$.
    Since $S \in \mathcal{P}_{n+1}$ and $K_{k,N}$ is a union of some subfamily of $\mathcal{P}_k$, it follows that $S \subseteq K_{k,N}$, because $\mathcal{P}_{n+1}$ is a refinement of $\mathcal{P}_k$ for any $k \le n$.
    By equation \eqref{eq:J_family} we conclude that  $J_{n+1,N} \subseteq \bigcup_{k=0}^{n} K_{k,N}$, which finishes the proof.
\end{proof}

\begin{lemma}\label{lem:disjoint}
    For any $N \ge 1$ the sets $K_{n,N}$ with $n \ge 0$ are disjoint.
\end{lemma}

\begin{proof}
    If $S \in \mathcal{J}_{n+1,N}$, then $\widehat{S} \in \mathcal{P}_n$, while $\bigcup_{k=0}^{n-1} K_{k,N}$ is a union of some subfamily of $ \bigcup_{k=0}^{n-1} \mathcal{P}_k$.
    Hence, either $S \subseteq\widehat{S} \subseteq \bigcup_{k=0}^{n-1} K_{k,N}$ or $\widehat{S} \cap \bigcup_{k=0}^{n-1} K_{k,N}= \emptyset$.
    So if $S \in \mathcal{J}_{n+1,N}$ and $S \cap \bigcup_{k=0}^{n-1} K_{k,N} = \emptyset$, then $\widehat{S} \cap \bigcup_{k=0}^{n-1} K_{k,N}= \emptyset$.
    By equation \eqref{eq:K_family} this implies that the sets $K_{n,N}$ and $\bigcup_{k=0}^{n-1} K_{k,N}$ are disjoint. The claim now follows.
\end{proof}

Lemma~\ref{lem:disjoint} implies that for every $N \ge 1$ we have a disjoint (double) union
\begin{equation}\label{eq:K_N}
    K_N=\bigcup_{n=0}^\infty K_{n,N} =\bigcup_{n=0}^\infty \bigcup_{A \in \KK_{n,N}} A.
\end{equation}
Next lemma shows that the sequence of sets $K_N$ is decreasing in $N$ with respect to inclusion. 

\begin{lemma}\label{lem:decreasing}
    $K_{N+1} \subseteq K_N$ for all $N \ge 1$. 
\end{lemma}

\begin{proof}
   Let $A \in \mathcal{K}_{n,N+1}$ for some $n \ge 0$. Then $A=\widehat{S}$ for some $S \in \mathcal{J}_{n+1,N+1}$.
   Clearly, $\mathcal{J}_{n+1,N+1} \subseteq \mathcal{J}_{n+1,N}$, so that $S \in \mathcal{J}_{n+1,N}$.
   From equation \eqref{eq:J_family} and Lemma~\ref{lem:J_n_N} it follows that $S \subseteq J_{n+1,N} \subseteq \bigcup_{k=0}^{n} K_{k,N}= \bigcup_{k=0}^{n} \bigcup_{B \in \KK_{k,N}} B$, so there exists $0 \le k \le n$ and $B \in \KK_{k,N}$ such that $S \cap B \ne \emptyset$. Since $S \in \mathcal{P}_{n+1}$ and $B \in \mathcal{P}_k$ with $k \le n$, we infer $S \subseteq B$ and consequently even $A=\widehat{S} \subseteq B$, because $k<n+1$. We have thus shown that $K_{n,N+1} \subseteq \bigcup_{k=0}^n K_{k,N}$, which implies $K_{N+1} \subseteq K_N$. 
\end{proof}

\subsection{Measure-theoretic properties of sets \texorpdfstring{$K_N$}{K\_N} and \texorpdfstring{$L_N$}{L\_N}}\label{subs:properties}\phantom{}\medskip

Note that all subsets of $\II$ considered above are open, so they are Borel measurable. We will now show that the Lebesgue measure $\lambda$ of $K_N$ is small.

\begin{lemma}\label{lem:lambda_K_N}
    $\lambda(K_N) \le \frac{2}{N}$ for all $N \ge 1$.
\end{lemma}

\begin{proof}
    By Lemma~\ref{lem:disjoint} the double union in \eqref{eq:K_N} is disjoint, so we have
\begin{equation*}
    \lambda(K_N)=\sum_{n=0}^\infty \sum_{A \in \KK_{n,N}} \lambda(A).
\end{equation*}
From equations \eqref{eq:K_family_0} and \eqref{eq:K_family} it follows that
\begin{equation*}
    \lambda(K_N)
    \le \sum_{S \in \mathcal{J}_{1,N}} \lambda(\widehat{S}) + \sum_{n=1}^\infty \sum_{\substack{S \in \mathcal{J}_{n+1,N} \\ S \cap \bigcup_{k=0}^{n-1} K_{k,N} = \emptyset}} \lambda(\widehat{S})= \sum_{S \in \mathcal{J}_{1,N}} 2\lambda(S) + \sum_{n=1}^\infty \sum_{\substack{S \in \mathcal{J}_{n+1,N} \\ S \cap \bigcup_{k=0}^{n-1} K_{k,N} = \emptyset}} 2\lambda(S).
\end{equation*}
Using property \eqref{eq:J_cond} we obtain
\begin{equation*}
    \lambda(K_N)
    \le \sum_{S \in \mathcal{J}_{1,N}} \frac{2\mu_Q^+(S \times \II)}{N \cdot \mu_Q^+(\II^2)} + \sum_{n=1}^\infty \sum_{\substack{S \in \mathcal{J}_{n+1,N} \\ S \cap \bigcup_{k=0}^{n-1} K_{k,N} = \emptyset}} \frac{2\mu_Q^+(S \times \II)}{N \cdot \mu_Q^+(\II^2)}.
\end{equation*}
By Lemma~\ref{lem:J_n_N} we have $\bigcup_{k=1}^n J_{k,N} \subseteq \bigcup_{k=0}^{n-1} K_{k,N}$, so we can further estimate
\begin{equation*}
\begin{aligned}
    \lambda(K_N)
    &\le \sum_{S \in \mathcal{J}_{1,N}} \frac{2\mu_Q^+(S \times \II)}{N \cdot \mu_Q^+(\II^2)}
    +\sum_{n=1}^\infty \sum_{\substack{S \in \mathcal{J}_{n+1,N} \\ S \cap \bigcup_{k=1}^n J_{k,N} = \emptyset}} \frac{2\mu_Q^+(S \times \II)}{N \cdot \mu_Q^+(\II^2)}\\
    &=\frac{2}{N \cdot \mu_Q^+(\II^2)} \cdot \sum_{\substack{S \in \mathcal{P}_1 \\ S \subseteq J_{1,N}}} \mu_Q^+(S \times \II)
    +\frac{2}{N \cdot \mu_Q^+(\II^2)} \cdot  \sum_{n=1}^\infty \sum_{\substack{S \in \mathcal{P}_{n+1} \\ S \subseteq J_{n+1,N} \setminus \bigcup_{k=1}^n J_{k,N}}} \mu_Q^+(S \times \II)\\
    & \le \frac{2}{N \cdot \mu_Q^+(\II^2)} \cdot \mu_Q^+(J_{1,N} \times \II)
    +\frac{2}{N \cdot \mu_Q^+(\II^2)} \cdot  \sum_{n=1}^\infty \mu_Q^+\Big(\Big(J_{n+1,N} \setminus \bigcup_{k=1}^n J_{k,N}\Big) \times \II\Big)\\
    &= \frac{2}{N \cdot \mu_Q^+(\II^2)} \cdot \mu_Q^+\Big(\bigcup_{n=1}^\infty J_{n,N} \times \II\Big) \le \frac{2}{N}.
\end{aligned}
\end{equation*}
\end{proof}

We can now give one of the key points of this construction. Let
\begin{equation}\label{eq:K}
    K=\bigcap_{N=1}^\infty K_N.
\end{equation}
The first crucial property of the set $K$ is that its Lebesgue measure is $0$.

\begin{lemma}\label{lem:lambda_K}
    We have $\lambda(K)=\displaystyle\lim_{N \to \infty} \lambda(K_N)=0$.
\end{lemma}

\begin{proof}
    By Lemma~\ref{lem:decreasing} the sequence of sets $K_N$ is decreasing in $N$. Hence, the conclusion follows directly from equation \eqref{eq:K} and Lemma~\ref{lem:lambda_K_N}.
\end{proof}

Furthermore, the restriction of measure $\mu_Q$ to the set $K \times \II$ is the zero measure.

\begin{lemma}\label{lem:mumu}
    The measures $\mu'$ defined for all $B \in\BB(\II^2)$ by
    \begin{equation*}
    \mu'(B)=\mu_Q\big(B \cap (K \times \II)\big) 
    \end{equation*}
    is the zero measure.
\end{lemma}

\begin{proof}
Let $(x,y) \in \II^2$ be arbitrary and let $R=[0,x] \times [0,y]$. Then
\begin{equation*}
    |\mu'(R)|=\big|\mu_Q\big(([0,x] \cap K)\times [0,y]\big)\big|=\big|\mu_Q\Big(\bigcap_{N=1}^\infty ([0,x] \cap K_N)\times [0,y]\Big)\big|.
\end{equation*}
Taking into account the monotonicity of $K_N$, see Lemma~\ref{lem:decreasing}, and the fact that $\mu_Q$ is a finite signed measure, we obtain
\begin{equation*}
    |\mu'(R)|=\big|\lim_{N \to \infty} \mu_Q\big(([0,x] \cap K_N)\times [0,y]\big)\big|.
\end{equation*}
Using equation \eqref{eq:K_N} and the fact that the union in \eqref{eq:K_N} is disjoint by Lemma~\ref{lem:disjoint}, we obtain
\begin{equation}\label{eq:abs}
\begin{aligned}
    |\mu'(R)| &=\Big|\lim_{N \to \infty} \sum_{n=0}^\infty \sum_{A \in \KK_{n,N}} \mu_Q\big(([0,x] \cap A)\times [0,y]\big)\Big|\\
    &\le \lim_{N \to \infty} \sum_{n=0}^\infty \sum_{A \in \KK_{n,N}}  \big|\mu_Q\big(([0,x] \cap A)\times [0,y]\big)\big|.
\end{aligned}
\end{equation}
For any rectangle $(a,b\} \times [0,y]$, where $\}$ stands for either $)$ or $]$, the continuity, groundedness and $1$-Lipschitz property of quasi-copulas imply
\begin{equation*}
    \big|\mu_Q\big((a,b\}\times [0,y]\big)\big|=|Q(b,y)-Q(a,y)|\le |b-a|=\lambda((a,b\}),
\end{equation*}
regardless of whether $(a,b\}$ is empty or not.
Applying this inequality to the rectangle $([0,x] \cap A)\times [0,y]$ in \eqref{eq:abs} and using also Lemma~\ref{lem:lambda_K} we obtain
\begin{equation*}
    |\mu'(R)| \le \lim_{N \to \infty} \sum_{n=0}^\infty \sum_{A \in \KK_{n,N}}  \lambda([0,x] \cap A)
    \le \lim_{N \to \infty} \sum_{n=0}^\infty \sum_{A \in \KK_{n,N}}  \lambda(A)= \lim_{N \to \infty} \lambda(K_N)=\lambda(K)=0
\end{equation*}
and consequently $\mu'(R)=0$.
To prove that $\mu'$ is the zero measure on $\BB(\II^2)$ we use a standard argument. By the above the collection
$\BB'=\{B \in \BB(\II^2) \mid \mu'(B)=0\}$ is a $\sigma$-algebra that contains all rectangles of the form $[0,x] \times [0,y]$. Hence, $\BB'=\BB(\II^2)$ and $\mu'$ is the zero measure.
\end{proof}

As a consequence to Lemma~\ref{lem:mumu}, we obtain the second crucial property of the set $K$, which will be essential for the convergence of measures constructed later.

\begin{lemma}\label{lem:abs_mu_Q}
    We have
    \begin{equation*}
    |\mu_Q|\big(K \times \II\big)=\displaystyle\lim_{N \to \infty} |\mu_Q|\big(K_N \times \II\big)=0.
    \end{equation*}
\end{lemma}

\begin{proof}
The sequences $K_N \times \II$ is decreasing in $N$ by Lemma~\ref{lem:decreasing}, and we have 
$\bigcap_{N=1}^\infty (K_N \times \II)=K \times \II$ by equation \eqref{eq:K}.
This implies
\begin{equation*}
    \lim_{N \to \infty} |\mu_Q|\big(K_N \times \II\big)=|\mu_Q|\big(K \times \II\big).
\end{equation*}

By Lemma~\ref{lem:mumu}, the measure $\mu'$ defined in that lemma is a zero measure.
Denote the support of $\mu_Q^+$ by $P^+$. Then
\begin{equation*}
    \mu_Q^+\big(K \times \II\big)=\mu_Q\big(P^+ \cap (K \times \II)\big)=\mu'(P^+)=0,
\end{equation*}
and similarly $\mu_Q^-\big(K \times \II)=0$. The claim now follows since $|\mu_Q|=\mu_Q^+ + \mu_Q^-$.

\end{proof}

Switching the first and second coordinate in the above constructions, we can analogously define for all $n \ge 0$ and $N \ge 1$ the sets
\begin{equation*}%\label{eq:I_family}
    \mathcal{I}_{n,N}=\Big\{S \in \mathcal{P}_n \mid 2^n \mu_Q^+(\II \times S) > N \cdot \mu_Q^+(\II^2)\Big\} \subseteq \mathcal{P}_n,\qquad I_{n,N}=\bigcup_{S \in \mathcal{I}_{n,N}} S,
\end{equation*}
\begin{equation*}%\label{eq:L_family_0}
    \LL_{0,N}=\big\{\widehat{S} \mid S \in \mathcal{I}_{1,N}\big\} \subseteq \mathcal{P}_0,\qquad L_{0,N}=\bigcup_{A \in \LL_{0,N}} A,
\end{equation*}
\begin{equation*}%\label{eq:L_family}
    \LL_{n,N}=\Big\{\widehat{S} \mid S \in \mathcal{I}_{n+1,N},\  S \cap \bigcup_{k=0}^{n-1} L_{k,N} = \emptyset\Big\} \subseteq \mathcal{P}_n,\qquad L_{n,N}=\bigcup_{A \in \LL_{n,N}} A,
\end{equation*}
\begin{equation}\label{eq:L_N}
    L_N=\bigcup_{n=0}^\infty L_{n,N} =\bigcup_{n=0}^\infty \bigcup_{A \in \LL_{n,N}} A, \qquad L=\bigcap_{N=1}^\infty L_N.
\end{equation}
By symmetry the versions of Lemmas~\ref{lem:lambda_K_N}--\ref{lem:abs_mu_Q} for the sets $L_N$ and $L$ also hold. In particular, $\lambda(L_N) \le \frac{2}{N}$ for all $N \ge 1$, $\lambda(L)=0$, the measure $\mu''$ defined for all $B \in \BB(\II^2)$ by
$$\mu''(B)=\mu_Q\big(B \cap (\II \times L)\big)$$
is the zero measure, and the following lemma holds.
\begin{lemma}\label{lem:abs_mu_Q_2}
    We have
    \begin{equation*}
    |\mu_Q|\big(\II \times L\big)=\displaystyle\lim_{N \to \infty} |\mu_Q|\big(\II \times L_N\big)=0.
    \end{equation*}
\end{lemma}

\subsection{Construction of quasi-copulas \texorpdfstring{$Q_N$}{Q\_N}}\label{subs:quasi-copulas}\phantom{}\medskip

We now turn our attention to the construction of a sequence of quasi-copulas $Q_N$.
For every positive integer $N$ denote the complements
\begin{equation*}
    K_N'=\II \setminus K_N \qquad\text{and}\qquad L_N'=\II \setminus L_N.
\end{equation*}
By equations \eqref{eq:K_N} and \eqref{eq:L_N} the sets $K_N$ and $L_N$ are open, so the sets $K_N'$ and $L_N'$ are closed and contain $0$ and $1$.
The restriction $Q|_{K_N' \times L_N'}$ of quasi-copula $Q$ to the set $K_N' \times L_N'$ is a sub-quasi-copula according to \cite[Definition~2.3]{QueSem05}. We can extend sub-quasi-copula $Q|_{K_N' \times L_N'}$ to a quasi-copula $Q_N$ using the formula from the proof of \cite[Theorem~2.3]{QueSem05}, which we recall now.
Let $(x,y)$ be an arbitrary point in $\II^2$ and denote
\begin{align*}
    x_1 &=\max \big([0,x] \cap K_N'\big), &\quad& x_2=\min \big([x,1] \cap K_N'\big), \\
    y_1 &=\max \big([0,y] \cap L_N'\big), &\quad& y_2=\min \big([y,1] \cap L_N'\big).
\end{align*}
Then $x_1 \le x \le x_2$ and $y_1 \le y \le y_2$. If $x \in K_N'$, then $x_1=x_2$, and if $y \in L_N'$, then $y_1=y_2$.
Now let
\begin{equation*}
    \alpha =\begin{cases}
        \frac{x-x_1}{x_2-x_1}, & x_1<x_2,\\
        1,& x_1=x_2,
    \end{cases}
    \quad\text{and}\quad
    \beta =\begin{cases}
        \frac{y-y_1}{y_2-y_1}, & y_1<y_2,\\
        1,& y_1=y_2.
    \end{cases}
\end{equation*}
Then the value of the extension at $(x,y)$ is given by
\begin{equation}\label{eq:Q_N}
\begin{aligned}
    Q_N(x,y) &=(1-\alpha)(1-\beta) Q(x_1,y_1)+(1-\alpha)\beta Q(x_1,y_2) \\
    &+\alpha(1-\beta) Q(x_2,y_1)+\alpha\beta Q(x_2,y_2).
\end{aligned}
\end{equation}
Note that $Q_N$ is bilinear on $[x_1,x_2] \times [y_1,y_2]$.
In particular, if $R$ is a closed rectangle such that only the vertices of $R$ lie in $K_N' \times L_N'$, then $Q_N$ is bilinear on $R$ (e.g. $R_5$ on Figure~\ref{fig:extension}).
If only the left and right side of $R$ lie in $K_N' \times L_N'$, then $Q_N$ is linear in $x$ on $R$ but not necessary in $y$ (e.g. $R_2$ and $R_8$ on Figure~\ref{fig:extension}),
if only the bottom and top side of $R$ lie in $K_N' \times L_N'$, then $Q_N$ is linear in $y$ on $R$ but not necessary in $x$ (e.g. $R_4$ and $R_6$ on Figure~\ref{fig:extension}),
and if the whole $R$ lies in $K_N' \times L_N'$, then $Q_N$ on $R$ is not linear in any coordinate in general (e.g. $R_1$, $R_3$, $R_7$, and $R_9$ on Figure~\ref{fig:extension}).
So the extension $Q_N$ is as linear as possible. This means that it spreads mass on $\II \setminus (K_N' \times L_N')=(K_N \times \II) \cup (\II \times L_N)$ uniformly in certain directions, which will be important later on.
\begin{figure}[ht]
    \centering
    \includegraphics{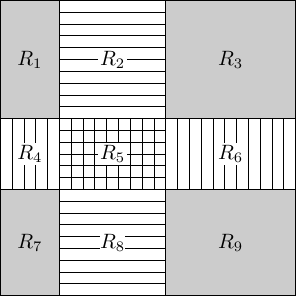}
    \caption{The extension $Q_N$ of the sub-quasi-copula $Q|_{K_N' \times L_N'}$ defined on the shaded region $K_N' \times L_N'$.
    The straight sections indicate in which direction the extension $Q_N$ is linear on $(K_N \times \II) \cup (\II \times L_N)$.}
    \label{fig:extension}
\end{figure}

\subsection{Construction of measures \texorpdfstring{$\mu_N$}{mu\_N}}\label{subs:measures}\phantom{}\medskip

We aim to prove that quasi-copulas $Q_N$ induce signed measures on $\BB(\II^2)$. We now construct these measures.
From equations \eqref{eq:K_N} and \eqref{eq:L_N} it follows that the sets $K_N$ and $L_N$ are countable unions of disjoint open intervals, say
\begin{equation}\label{eq:inter}
    K_N= \bigcup_{i=0}^\infty (a_i^N,b_i^N) \quad\text{and}\quad L_N =\bigcup_{j=0}^\infty (c_i^N,d_i^N),
\end{equation}
where each $(a_i^N,b_i^N)$ and $(c_i^N,d_i^N)$ is a member of $\bigcup_{n=0}^\infty \mathcal{P}_n$.
We are assuming here without loss of generality that the unions are infinite. If they are finite the arguments are essentially the same.

For a set $E \in \BB(\II)$ with $\lambda(E)>0$, define a probability measure $\lambda_E$ on $\BB(\II)$ by
\begin{equation*}
    \lambda_E(A)=\lambda(A \cap E)/\lambda(E)
\end{equation*}
for all $A \in \BB(\II)$, and finite signed measures $\varphi_E$ and $\psi_E$ on $\BB(\II)$ by
\begin{equation}\label{eq:phi_psi}
    \varphi_E(A)=\mu_Q\big(E \times (A \cap L_N')\big) \qquad\text{and}\qquad \psi_E(A)=\mu_Q\big((A \cap K_N') \times E\big)
\end{equation}
for all $A \in \BB(\II)$.
Furthermore, introduce finite signed measures $\muhat_N$ and $\muover_N$ on $\BB(\II^2)$ by
\begin{equation}\label{eq:muhat_N}
\begin{aligned}
    \muhat_N
    &=\sum_{i=0}^\infty \lambda_{(a_i^N,b_i^N)} \times \varphi_{(a_i^N,b_i^N)}
    +\sum_{j=0}^\infty \psi_{(c_j^N,d_j^N)} \times \lambda_{(c_j^N,d_j^N)}\\
    &\hspace{0.5cm}+\sum_{i=0}^\infty \sum_{j=0}^\infty \mu_Q\big((a_i^N,b_i^N)\times(c_j^N,d_j^N)\big) \cdot \lambda_{(a_i^N,b_i^N)} \times \lambda_{(c_j^N,d_j^N)}
\end{aligned}
\end{equation}
and
\begin{equation}\label{eq:muover_N}
    \muover_N(B)=\mu_Q\big(B \cap (K_N' \times L_N')\big)
\end{equation}
for all $B \in \BB(\II^2)$, where $\times$ denotes both the product of measures and the Cartesian product of sets.
The idea behind the definition of $\muhat_N$ is that it should correspond to the measure induced by the the extension $Q_N$ on the set $(K_N \times \II) \cup (\II \times L_N)$.
In particular, the first sum corresponds to the regions where $Q_N$ is linear in the $x$ coordinate (regions $R_2$ and $R_8$ on Figure~\ref{fig:extension}), the second sum corresponds to the regions where $Q_N$ is linear in $y$ coordinate (regions $R_4$ and $R_6$ on Figure~\ref{fig:extension}) and the third sum corresponds to the regions where $Q_N$ is bilinear (region $R_5$ on Figure~\ref{fig:extension}). 

We first need to show that $\muhat_N$ is well defined.

\begin{lemma}\label{lem:muhat_N}
    For every $N \ge 1$ the sum in equation \eqref{eq:muhat_N} converges in the total variation norm, so $\muhat_N$ is a finite signed measure on $\BB(\II^2)$.
\end{lemma}

\begin{proof}
For every $E \in \BB(\II)$ with $\lambda(E)>0$ let $\varphi_E=\varphi_E^+-\varphi_E^-$ be the Jordan decomposition of $\varphi_E$ and denote the supports of $\varphi_E^+$ and $\varphi_E^-$ by $P_E^+$ and $P_E^-$, respectively.
Since $\lambda_E$ is a probability measure, and $\lambda_E \times \varphi_E^+$ and $\lambda_E \times \varphi_E^-$ are positive measures, we can estimate using equation \eqref{eq:TV}
\begin{equation}\label{eq:times1}
\begin{aligned}
    \|\lambda_E \times \varphi_E\|_{TV} &=\|\lambda_E \times \varphi_E^+ - \lambda_E \times \varphi_E^-\|_{TV}\\
    &\le \|\lambda_E \times \varphi_E^+\|_{TV} + \|\lambda_E \times \varphi_E^-\|_{TV}
    =\varphi_E^+(\II)+\varphi_E^-(\II)=\|\varphi_E\|_{TV}.
\end{aligned}
\end{equation}
Using equations \eqref{eq:TV} and \eqref{eq:phi_psi} we estimate
\begin{equation}\label{eq:times2}
\begin{aligned}
    \|\varphi_E\|_{TV}
    &=\varphi_E(P_E^+)-\varphi_E(P_E^-)\\
    &=\mu_Q\big(E \times (P_E^+ \cap L_N')\big)- \mu_Q\big(E \times (P_E^- \cap L_N')\big)\\
    &\le \mu_Q^+\big(E \times (P_E^+ \cap L_N')\big)+ \mu_Q^-\big(E \times (P_E^- \cap L_N')\big)\\
    &\le \mu_Q^+(E \times L_N')+\mu_Q^-(E \times L_N')
    =|\mu_Q|(E \times L_N').
\end{aligned}
\end{equation}
Combining inequalities \eqref{eq:times1} and \eqref{eq:times2} gives
\begin{equation*}
    \|\lambda_E \times \varphi_E\|_{TV} \le |\mu_Q|(E \times L_N'),
\end{equation*}
and similarly we obtain
\begin{equation*}
    \|\psi_E \times \lambda_E\|_{TV} \le |\mu_Q|(K_N' \times E).
\end{equation*}
These two inequalities imply
\begin{equation}\label{eq:norm}
\begin{aligned}
    &\sum_{i=0}^\infty \|\lambda_{(a_i^N,b_i^N)} \times \varphi_{(a_i^N,b_i^N)}\|_{TV}
    +\sum_{j=0}^\infty \|\psi_{(c_j^N,d_j^N)} \times \lambda_{(c_j^N,d_j^N)}\|_{TV}\\
    &\hspace{0.5cm}+\sum_{i=0}^\infty \sum_{j=0}^\infty \big\|\mu_Q\big((a_i^N,b_i^N)\times(c_j^N,d_j^N)\big) \cdot\lambda_{(a_i^N,b_i^N)} \times \lambda_{(c_j^N,d_j^N)}\big\|_{TV}\\
    &\le \sum_{i=0}^\infty |\mu_Q|\big((a_i^N,b_i^N) \times L_N'\big)
    +\sum_{j=0}^\infty |\mu_Q|\big(K_N' \times (c_j^N,d_j^N)\big)+\sum_{i=0}^\infty \sum_{j=0}^\infty \big|\mu_Q\big((a_i^N,b_i^N)\times(c_j^N,d_j^N)\big)\big|\\
    &\le \sum_{i=0}^\infty |\mu_Q|\big((a_i^N,b_i^N) \times L_N'\big)
    +\sum_{j=0}^\infty |\mu_Q|\big(K_N' \times (c_j^N,d_j^N)\big)
    +\sum_{i=0}^\infty \sum_{j=0}^\infty |\mu_Q|\big((a_i^N,b_i^N)\times(c_j^N,d_j^N)\big)\\
    &=|\mu_Q|(K_N \times L_N')+|\mu_Q|(K_N' \times L_N)+|\mu_Q|(K_N \times L_N)\\
    &=|\mu_Q|\big((K_N \times \II) \cup (\II \times L_N)\big)< \infty.
\end{aligned}
\end{equation}
Since the vector space of all finite measures equipped with the total variation norm is a Banach space, this implies that the sum in equation \eqref{eq:muhat_N} converges in the total variation norm and $\muhat_N$ is a finite signed measure.
\end{proof}

We can now prove that quasi-copulas $Q_N$ induce signed measures.
For every $N \ge 1$ define a finite signed measure
\begin{equation*}
    \mu_N=\muhat_N+\muover_N,
\end{equation*}
where $\muhat_N$ and $\muover_N$ are given in equations \eqref{eq:muhat_N} and \eqref{eq:muover_N}. We show with a direct calculation that measure $\mu_N$ is induced by $Q_N$.

\begin{lemma}\label{lem:measure_inducing}
    For every $N \ge 1$ quasi-copula $Q_N$ induces signed measure $\mu_N$.
\end{lemma}

\begin{proof}
For any $(x,y) \in \II^2$ we use equations \eqref{eq:muhat_N} and \eqref{eq:muover_N} to get
\begin{equation}\label{eq:mu_N}
\begin{aligned}
    &\mu_N\big([0,x] \times [0,y]\big)\\
    &=\sum_{i=0}^\infty \frac{\lambda\big([0,x] \cap (a_i^N,b_i^N)\big)}{\lambda((a_i^N,b_i^N))} \cdot \mu_Q\big((a_i^N,b_i^N) \times ([0,y] \cap L_N')\big)\\
    &\hspace{0.5cm}+\sum_{j=0}^\infty \mu_Q\big(([0,x] \cap K_N') \times (c_j^N,d_j^N)\big) \cdot \frac{\lambda\big([0,y]\cap (c_j^N,d_j^N)\big)}{\lambda((c_j^N,d_j^N))}\\
    &\hspace{0.5cm}+\sum_{i=0}^\infty \sum_{j=0}^\infty \mu_Q\big((a_i^N,b_i^N)\times(c_j^N,d_j^N)\big) \cdot \frac{\lambda\big([0,x] \cap (a_i^N,b_i^N)\big)}{\lambda((a_i^N,b_i^N))} \cdot \frac{\lambda\big([0,y]\cap (c_j^N,d_j^N)\big)}{\lambda((c_j^N,d_j^N))}\\
    &\hspace{0.5cm}+\mu_Q\big(([0,x] \times [0,y]) \cap (K_N' \times L_N')\big).
\end{aligned}
\end{equation}

We consider four cases, depending on where the point $(x,y)$ lies.

\medskip{\bf Case 1:} $x \notin K_N$ and $y \notin L_N$. On Figure~\ref{fig:extension} this corresponds to $(x,y) \in R_1 \cup R_3 \cup R_7 \cup R_9$.
In this case $x$ is not contained in any $(a_i^N,b_i^N)$ and $y$ is not contained in any $(c_j^N,d_j^N)$ so equation \eqref{eq:mu_N} becomes
\begin{align*}
    &\mu_N\big([0,x] \times [0,y]\big)\\
    &=\sum_{(a_i^N,b_i^N) \subseteq [0,x]} \mu_Q\big((a_i^N,b_i^N) \times ([0,y] \cap L_N')\big)+\sum_{(c_j^N,d_j^N) \subseteq [0,y]} \mu_Q\big(([0,x] \cap K_N') \times (c_j^N,d_j^N)\big)\\
    &\hspace{0.5cm}+\sum_{(a_i^N,b_i^N) \subseteq [0,x]} \ \sum_{(c_j^N,d_j^N) \subseteq [0,y]} \mu_Q\big((a_i^N,b_i^N)\times(c_j^N,d_j^N)\big)\\
    &\hspace{0.5cm}+\mu_Q\big(([0,x] \times [0,y]) \cap (K_N' \times L_N')\big)\\
    &=\mu_Q\big(([0,x] \cap K_N) \times ([0,y] \cap L_N')\big) + \mu_Q\big(([0,x] \cap K_N') \times ([0,y] \cap L_N)\big)\\
    &\hspace{0.5cm}+\mu_Q\big(([0,x] \cap K_N) \times ([0,y] \cap L_N)\big)
    +\mu_Q\big(([0,x] \cap K_N') \times ([0,y] \cap L_N')\big)\\
    &=\mu_Q\big(([0,x] \times [0,y]) \cap (K_N \times L_N')\big) + \mu_Q\big(([0,x] \times [0,y]) \cap (K_N' \times L_N)\big)\\
    &\hspace{0.5cm}+\mu_Q\big(([0,x] \times [0,y]) \cap (K_N \times L_N)\big)
    +\mu_Q\big(([0,x] \times [0,y]) \cap (K_N' \times L_N')\big).
\end{align*}
Since
\begin{equation*}
    (K_N \times L_N') \cup (K_N' \times L_N) \cup (K_N \times L_N) \cup (K_N' \times L_N')=\II^2
\end{equation*}
and this union is disjoint, we conclude
\begin{equation}\label{eq:C1}
\mu_N\big([0,x] \times [0,y]\big)=\mu_Q\big([0,x] \times [0,y]\big)=Q(x,y).
\end{equation}
By equation \eqref{eq:Q_N} this is equal to $Q_N(x,y)$ because $x \in K_N'$ and $y \in L_N'$.

\medskip{\bf Case 2:} $x \in K_N$ and $y \notin L_N$. On Figure~\ref{fig:extension} this corresponds to $(x,y) \in R_2 \cup R_8$.
We may assume without loss of generality that $x \in (a_0^N,b_0^N)$ since the order of the intervals in equation \eqref{eq:inter} is arbitrary.
On the other hand, $y$ is not contained in any $(c_j^N,d_j^N)$.
By splitting $[0,x] \times [0,y] = \big([0,a_0^N] \times [0,y]\big) \cup \big((a_0^N,x] \times [0,y]\big)$, we infer that equation \eqref{eq:mu_N} can be written as
\begin{align*}
    \mu_N\big([0,x] \times [0,y]\big)
    &=\mu_N\big([0,a_0^N] \times [0,y]\big)\\
    &\hspace{0.5cm}+\frac{\lambda((a_0^N,x])}{\lambda((a_0^N,b_0^N))} \cdot \mu_Q\big((a_0^N,b_0^N) \times \big([0,y] \cap L_N'\big)\big)\\
    &\hspace{0.5cm}+\sum_{(c_j^N,d_j^N) \subseteq [0,y]} \mu_Q\big(\big((a_0^N,x] \cap K_N'\big) \times (c_j^N,d_j^N)\big)\\
    &\hspace{0.5cm}+\sum_{(c_j^N,d_j^N) \subseteq [0,y]} \mu_Q\big((a_0^N,b_0^N)\times(c_j^N,d_j^N)\big) \cdot \frac{\lambda((a_0^N,x])}{\lambda((a_0^N,b_0^N))}\\
    &\hspace{0.5cm}+\mu_Q\big(\big((a_0^N,x] \times [0,y]\big) \cap (K_N' \times L_N')\big)
\end{align*}
Note that $a_0^N \notin K_N$, so we may use Case 1 to evaluate the first term in the above expression. Using equation \eqref{eq:C1}, the first term is equal to $\mu_N\big([0,a_0^N] \times [0,y]\big)=Q(a_0^N,y)$. Furthermore, the last term and the first of the two sums $\sum_{(c_j^N,d_j^N) \subseteq [0,y]}$ are equal to $0$ because $(a_0^N,x] \subseteq K_N$.
Hence,
\begin{equation}\label{eq:C2}
\begin{aligned}
    \mu_N\big([0,x] \times [0,y]\big)
    &=Q(a_0^N,y)
    +\frac{x-a_0^N}{b_0^N-a_0^N} \cdot \mu_Q\big((a_0^N,b_0^N) \times ([0,y] \cap L_N')\big)\\
    &\hspace{0.5cm}+\frac{x-a_0^N}{b_0^N-a_0^N} \cdot \sum_{(c_j^N,d_j^N) \subseteq [0,y]} \mu_Q\big((a_0^N,b_0^N)\times(c_j^N,d_j^N)\big)\\
    &=Q(a_0^N,y)
    +\frac{x-a_0^N}{b_0^N-a_0^N} \cdot \mu_Q\big((a_0^N,b_0^N) \times ([0,y] \cap L_N')\big)\\
    &\hspace{0.5cm}+\frac{x-a_0^N}{b_0^N-a_0^N} \cdot \mu_Q\big((a_0^N,b_0^N)\times ([0,y] \cap L_N)\big)\\
    &=Q(a_0^N,y)
    +\frac{x-a_0^N}{b_0^N-a_0^N} \cdot \mu_Q\big((a_0^N,b_0^N) \times [0,y]\big).
\end{aligned}
\end{equation}
Finally, by continuity of $Q$ we can express
\begin{align*}
    \mu_N\big([0,x] \times [0,y]\big)
    &=Q(a_0^N,y)
    +\frac{x-a_0^N}{b_0^N-a_0^N} \big(Q(b_0^N,y)-Q(a_0^N,y)\big)\\
    &=\frac{b_0^N-x}{b_0^N-a_0^N} Q(a_0^N,y)+\frac{x-a_0^N}{b_0^N-a_0^N} Q(b_0^N,y).
\end{align*}
By equation \eqref{eq:Q_N} this is equal to $Q_N(x,y)$ because $x \in K_N$ and $y \in L_N'$.

\medskip{\bf Case 3:} $x \notin K_N$ and $y \in L_N$. On Figure~\ref{fig:extension} this corresponds to $(x,y) \in R_4 \cup R_6$.
This case is treated similarly as Case 2. Assuming $y \in (c_0^N,d_0^N)$, we obtain
\begin{equation*}
    \mu_N\big([0,x] \times [0,y]\big)=\frac{d_0^N-y}{d_0^N-c_0^N} Q(x,c_0^N)+\frac{y-c_0^N}{d_0^N-c_0^N} Q(x,d_0^N),
\end{equation*}
which is again equal to $Q_N(x,y)$ by equation \eqref{eq:Q_N}.

\medskip{\bf Case 4:} $x \in K_N$ and $y \in L_N$. On Figure~\ref{fig:extension} this corresponds to $(x,y) \in R_5$.
We may assume without loss of generality that $x \in (a_0^N,b_0^N)$ and $y \in (c_0^N,d_0^N)$. By splitting $[0,x] \times [0,y] = ([0,x] \times [0,c_0^N]) \cup ([0,x] \times (c_0^N,y])$, we infer that equation \eqref{eq:mu_N} can be written as
\begin{align*}
    \mu_N\big([0,x] \times [0,y]\big)
    &=\mu_N\big([0,x] \times [0,c_0^N]\big)\\
    &\hspace{0.5cm}+\sum_{i=0}^\infty \frac{\lambda\big([0,x] \cap (a_i^N,b_i^N)\big)}{\lambda((a_i^N,b_i^N))} \cdot \mu_Q\big((a_i^N,b_i^N) \times \big((c_0^N,y] \cap L_N'\big)\big)\\
    &\hspace{0.5cm}+\mu_Q\big(\big([0,x] \cap K_N'\big) \times (c_0^N,d_0^N)\big) \cdot \frac{\lambda((c_0^N,y])}{\lambda((c_0^N,d_0^N))}\\
    &\hspace{0.5cm}+\sum_{i=0}^\infty \mu_Q\big((a_i^N,b_i^N)\times(c_0^N,d_0^N)\big) \cdot \frac{\lambda\big([0,x] \cap (a_i^N,b_i^N)\big)}{\lambda((a_i^N,b_i^N))} \cdot \frac{\lambda((c_0^N,y])}{\lambda((c_0^N,d_0^N))}\\
    &\hspace{0.5cm}+\mu_Q\big(\big([0,x] \times (c_0^N,y]\big) \cap (K_N' \times L_N')\big).
\end{align*}
Note that $c_0^N \notin L_N$, so by Case 2, using equation \eqref{eq:C2}, the first term in the above expression is equal to
\begin{equation*}
    \mu_N\big([0,x] \times [0,c_0^N]\big)=Q(a_0^N,c_0^N)+\frac{x-a_0^N}{b_0^N-a_0^N} \cdot \mu_Q\big((a_0^N,b_0^N) \times [0,c_0^N]\big).
\end{equation*}
Furthermore, the last term and the first of the two sums $\sum_{i=0}^\infty$ are equal to $0$ because $(c_0^N,y] \subseteq L_N$. Hence,
\begin{align*}
    \mu_N\big([0,x] \times [0,y]\big)
    &=Q(a_0^N,c_0^N)+\frac{x-a_0^N}{b_0^N-a_0^N} \cdot \mu_Q\big((a_0^N,b_0^N) \times [0,c_0^N]\big)\\
    &\hspace{0.5cm}+\frac{y-c_0^N}{d_0^N-c_0^N} \cdot \mu_Q\big(([0,x] \cap K_N') \times (c_0^N,d_0^N)\big)\\
    &\hspace{0.5cm}+\frac{y-c_0^N}{d_0^N-c_0^N}\cdot \sum_{i=0}^\infty \mu_Q\big((a_i^N,b_i^N)\times(c_0^N,d_0^N)\big) \cdot \frac{\lambda\big([0,x] \cap (a_i^N,b_i^N)\big)}{\lambda((a_i^N,b_i^N))}.
\end{align*}
We can simplify the last sum, also using $x \in (a_0^N,b_0^N)$, to obtain
\begin{equation}\label{eq:intermediate}
\begin{aligned}
    \mu_N\big([0,x] \times [0,y]\big)
    &=Q(a_0^N,c_0^N)+\frac{x-a_0^N}{b_0^N-a_0^N} \cdot \mu_Q\big((a_0^N,b_0^N) \times [0,c_0^N]\big)\\
    &\hspace{0.5cm}+\frac{y-c_0^N}{d_0^N-c_0^N} \cdot \mu_Q\big(([0,x] \cap K_N') \times (c_0^N,d_0^N)\big)\\
    &\hspace{0.5cm}+\frac{y-c_0^N}{d_0^N-c_0^N}\cdot \sum_{(a_i^N,b_i^N) \subseteq [0,x]} \mu_Q\big((a_i^N,b_i^N)\times(c_0^N,d_0^N)\big)\\
    &\hspace{0.5cm}+\frac{x-a_0^N}{b_0^N-a_0^N} \cdot \frac{y-c_0^N}{d_0^N-c_0^N}\cdot \mu_Q\big((a_0^N,b_0^N)\times(c_0^N,d_0^N)\big).
\end{aligned}
\end{equation}
Note that
\begin{equation*}
\bigcup_{(a_i^N,b_i^N) \subseteq [0,x]} (a_i^N,b_i^N)=[0,x] \cap \big(K_N \setminus (a_0^N,x]\big) \qquad\text{and}\qquad K_N' \cup \big(K_N \setminus (a_0^N,x]\big)=\II \setminus (a_0^N,x],
\end{equation*}
so we may combine the second and third row of equation \eqref{eq:intermediate} to get
\begin{align*}
    \mu_N\big([0,x] \times [0,y]\big)
    &=Q(a_0^N,c_0^N)+\frac{x-a_0^N}{b_0^N-a_0^N} \cdot \mu_Q\big((a_0^N,b_0^N) \times [0,c_0^N]\big)\\
    &\hspace{0.5cm}+\frac{y-c_0^N}{d_0^N-c_0^N} \cdot \mu_Q\big([0,a_0^N] \times (c_0^N,d_0^N)\big)\\
    &\hspace{0.5cm}+\frac{x-a_0^N}{b_0^N-a_0^N} \cdot \frac{y-c_0^N}{d_0^N-c_0^N}\cdot \mu_Q\big((a_0^N,b_0^N)\times(c_0^N,d_0^N)\big).
\end{align*}
The continuity of $Q$ implies
\begin{align*}
    \mu_Q\big((a_0^N,b_0^N) \times [0,c_0^N])\big) &=Q(b_0^N,c_0^N)-Q(a_0^N,c_0^N),\\
    \mu_Q\big([0,a_0^N] \times (c_0^N,d_0^N)\big) &=Q(a_0^N,d_0^N)-Q(a_0^N,c_0^N),\\
    \mu_Q\big((a_0^N,b_0^N)\times(c_0^N,d_0^N)\big) &=Q(b_0^N,d_0^N)-Q(a_0^N,d_0^N)-Q(b_0^N,c_0^N)+Q(a_0^N,c_0^N),
\end{align*}
so that
\begin{align*}
    &\mu_N\big([0,x] \times [0,y]\big)\\
    &=Q(a_0^N,c_0^N)+\frac{x-a_0^N}{b_0^N-a_0^N} \cdot \big(Q(b_0^N,c_0^N)-Q(a_0^N,c_0^N)\big)\\
    &\hspace{0.5cm}+\frac{y-c_0^N}{d_0^N-c_0^N} \cdot \big(Q(a_0^N,d_0^N)-Q(a_0^N,c_0^N)\big)\\
    &\hspace{0.5cm}+\frac{x-a_0^N}{b_0^N-a_0^N} \cdot \frac{y-c_0^N}{d_0^N-c_0^N}\cdot \big(Q(b_0^N,d_0^N)-Q(a_0^N,d_0^N)-Q(b_0^N,c_0^N)+Q(a_0^N,c_0^N)\big)\\
    &=\frac{b_0^N-x}{b_0^N-a_0^N} \cdot \frac{d_0^N-y}{d_0^N-c_0^N} \cdot Q(a_0^N,c_0^N)
    +\frac{b_0^N-x}{b_0^N-a_0^N} \cdot \frac{y-c_0^N}{d_0^N-c_0^N}\cdot Q(a_0^N,d_0^N)\\
    &\hspace{0.5cm}+\frac{x-a_0^N}{b_0^N-a_0^N} \cdot \frac{d_0^N-y}{d_0^N-c_0^N} \cdot Q(b_0^N,c_0^N)
    +\frac{x-a_0^N}{b_0^N-a_0^N} \cdot \frac{y-c_0^N}{d_0^N-c_0^N}\cdot Q(b_0^N,d_0^N).
\end{align*}
By equation \eqref{eq:Q_N} this is again equal to $Q_N(x,y)$ because $x \in K_N$ and $y \in L_N$.
\end{proof}

Next step is to establish the convergence of finite signed measures $\mu_N$.

\begin{lemma}\label{lem:TV_convergence}
    The sequence of measures $\mu_N$ converges to $\mu_Q$ in the total variation norm.
\end{lemma}

\begin{proof}
    From the proof of Lemma~\ref{lem:muhat_N}, see in particular inequality \eqref{eq:norm}, it follows that
    \begin{equation*}
        \|\muhat_N\|_{TV}\le |\mu_Q|\big((K_N \times \II) \cup (\II \times L_N)\big) \le |\mu_Q|\big(K_N \times \II\big)+|\mu_Q|\big(\II \times L_N\big).
    \end{equation*}
    When $N$ tends to infinity, the right-hand side converges to $0$ by Lemmas~\ref{lem:abs_mu_Q} and \ref{lem:abs_mu_Q_2}, so the sequence of measures $\muhat_N$ converges to the zero measure in the total variation norm.
    Furthermore, note that
    \begin{equation*}
        (\mu_Q-\muover_N)(B)=\mu_Q(B)-\mu_Q\big(B \cap (K_N' \times L_N')\big)=\mu_Q\big(B \cap \big((K_N \times \II) \cup (\II \times L_N)\big)\big)
    \end{equation*}
    for all $B \in \BB(\II^2)$. So we can use equation \eqref{eq:TV} and an analogous calculation as in inequality \eqref{eq:times2} to obtain
    \begin{equation*}
        \|\mu_Q-\muover_N\|_{TV}
        \le |\mu_Q|\big((K_N \times \II) \cup (\II \times L_N)\big).
    \end{equation*}
    Using a similar argument as for $\muhat_N$, we infer that the sequence $\mu_Q-\muover_N$ converges to the zero measure, so that $\muover_N$ converges to $\mu_Q$ in the total variation norm.
    We conclude that the sequence of finite measures $\mu_N=\muhat_N+\muover_N$ converges to $\mu_Q$.
\end{proof}

\subsection{Decomposition of quasi-copulas \texorpdfstring{$Q_N$}{Q\_N}}\label{subs:decomposition}\phantom{}\medskip

We have so far shown that quasi-copulas $Q_N$ induce signed measures $\mu_N$ and the measures $\mu_N$ converge in the total variation norm to measure $\mu_Q$, which is induced by $Q$.

The final property of quasi-copulas $Q_N$ that we need is that every $Q_N$ is a linear combination of two copulas. To prove this we will show that each $Q_N$ satisfies condition~$(ii)$ of Theorem~\ref{thm:LCC}.

\begin{lemma}\label{lem:combination_of_copulas}
    For every $N \ge 1$ there exist bivariate copulas $A_N$ and $B_N$ and real numbers $\alpha_N$ and $\beta_N$ such that $Q_N=\alpha_N A_N+\beta_N B_N$. Consequently, $\mu_N=\alpha_N \mu_{A_N}+\beta_N \mu_{B_N}$.
\end{lemma}

\begin{proof}
Fix $N \ge 1$. By Theorem~\ref{thm:LCC} it suffices to prove that 
\begin{equation*}
        \alpha_{Q_N}=\displaystyle\sup_{n \geq 1} \left\{ \max_{k \in [2^n]} 2^n \sum_{l=1}^{2^n} V_{Q_N}(R_{kl}^{n})^+ \, , \, \max_{l \in [2^n]} 2^n \sum_{k=1}^{2^n} V_{Q_N}(R_{kl}^{n})^+ \right\}<\infty,
\end{equation*}
where $R_{kl}^{n}=[\tfrac{k-1}{2^n},\tfrac{k}{2^n}] \times [\tfrac{l-1}{2^n},\tfrac{l}{2^n}]$ for all $k,l \in [2^n]$ and $x^+=\max\{x,0\}$ for all $x \in \RR$.
Let $n \ge 1$ and $k \in [2^n]$. The continuity of $Q_N$ implies
\begin{equation*}
V_{Q_N}(R_{kl}^{n})^+ = \mu_N(R_{kl}^n)^+=\mu_N\big((\tfrac{k-1}{2^n},\tfrac{k}{2^n}) \times (\tfrac{l-1}{2^n},\tfrac{l}{2^n})\big)^+ \le \mu_N^+\big((\tfrac{k-1}{2^n},\tfrac{k}{2^n}) \times (\tfrac{l-1}{2^n},\tfrac{l}{2^n})\big),
\end{equation*}
therefore
\begin{equation}\label{eq:sum_V_Q_N}
    2^n \sum_{l=1}^{2^n} V_{Q_N}(R_{kl}^{n})^+ \le 2^n \sum_{l=1}^{2^n} \mu_N^+\big((\tfrac{k-1}{2^n},\tfrac{k}{2^n}) \times (\tfrac{l-1}{2^n},\tfrac{l}{2^n})\big)\le 2^n\mu_N^+\big((\tfrac{k-1}{2^n},\tfrac{k}{2^n}) \times \II\big).
\end{equation}
Equation  $\lambda_E \times \varphi_E =\lambda_E \times \varphi_E^+ - \lambda_E \times \varphi_E^-$ implies $(\lambda_E \times \varphi_E)^+ \le \lambda_E \times \varphi_E^+$ for all $E \in \BB(\II)$. Similarly, $(\psi_E \times \lambda_E)^+ \le \psi_E^+ \times \lambda_E$ for all $E \in \BB(\II)$. Hence, by equations \eqref{eq:muhat_N} and \eqref{eq:muover_N},
\begin{align*}
    \mu_N^+ \le \muhat_N^+ + \muover_N^+
    & \le \sum_{i=0}^\infty \lambda_{(a_i^N,b_i^N)} \times \varphi_{(a_i^N,b_i^N)}^+
    +\sum_{j=0}^\infty \psi_{(c_j^N,d_j^N)}^+ \times \lambda_{(c_j^N,d_j^N)}\\
    &\hspace{0.5cm}+\sum_{i=0}^\infty \sum_{j=0}^\infty \mu_Q^+\big((a_i^N,b_i^N)\times(c_j^N,d_j^N)\big) \cdot \lambda_{(a_i^N,b_i^N)} \times \lambda_{(c_j^N,d_j^N)}\\
    &\hspace{0.5cm}+ \muover_N^+,
\end{align*}
where we also used $\mu_Q\big((a_i^N,b_i^N)\times(c_j^N,d_j^N)\big)^+ \le \mu_Q^+\big((a_i^N,b_i^N)\times(c_j^N,d_j^N)\big)$. This implies
\begin{align*}
    \mu_N^+\big((\tfrac{k-1}{2^n},\tfrac{k}{2^n}) \times \II\big)
    & \le \sum_{i=0}^\infty \lambda_{(a_i^N,b_i^N)}((\tfrac{k-1}{2^n},\tfrac{k}{2^n})) \cdot \varphi_{(a_i^N,b_i^N)}^+(\II)
    +\sum_{j=0}^\infty \psi_{(c_j^N,d_j^N)}^+((\tfrac{k-1}{2^n},\tfrac{k}{2^n})) \cdot \lambda_{(c_j^N,d_j^N)}(\II)\\
    &\hspace{0.5cm}+\sum_{i=0}^\infty \sum_{j=0}^\infty \mu_Q^+\big((a_i^N,b_i^N)\times(c_j^N,d_j^N)\big) \cdot \lambda_{(a_i^N,b_i^N)}((\tfrac{k-1}{2^n},\tfrac{k}{2^n})) \cdot \lambda_{(c_j^N,d_j^N)}(\II)\\
    &\hspace{0.5cm}+ \muover_N^+\big((\tfrac{k-1}{2^n},\tfrac{k}{2^n}\big) \times \II).
\end{align*}
From equations \eqref{eq:phi_psi} and \eqref{eq:muover_N} we infer
\begin{align*}
    \varphi_{(a_i^N,b_i^N)}^+(\II) & \le \mu_Q^+\big((a_i^N,b_i^N) \times L_N'\big),\\
    \psi_{(c_j^N,d_j^N)}^+((\tfrac{k-1}{2^n},\tfrac{k}{2^n})) & \le \mu_Q^+\big(\big((\tfrac{k-1}{2^n},\tfrac{k}{2^n}) \cap K_N'\big) \times (c_j^N,d_j^N)\big),\\
    \muover_N^+\big((\tfrac{k-1}{2^n},\tfrac{k}{2^n}) \times \II\big) & \le \mu_Q^+\big(\big((\tfrac{k-1}{2^n},\tfrac{k}{2^n}) \times \II\big) \cap (K_N' \times L_N')\big)
    =\mu_Q^+\big(\big((\tfrac{k-1}{2^n},\tfrac{k}{2^n}) \cap K_N'\big) \times L_N'\big),
\end{align*}
therefore
\begin{equation}\label{eq:est_alpha}
\begin{aligned}
    &\mu_N^+\big((\tfrac{k-1}{2^n},\tfrac{k}{2^n}) \times \II\big)\\
    &\le \sum_{i=0}^\infty \lambda_{(a_i^N,b_i^N)}\big((\tfrac{k-1}{2^n},\tfrac{k}{2^n})\big) \cdot \mu_Q^+\big((a_i^N,b_i^N) \times L_N'\big)
    +\sum_{j=0}^\infty \mu_Q^+\big(\big((\tfrac{k-1}{2^n},\tfrac{k}{2^n}) \cap K_N'\big) \times (c_j^N,d_j^N)\big)\\
    &\hspace{0.5cm}+\sum_{i=0}^\infty \lambda_{(a_i^N,b_i^N)}\big((\tfrac{k-1}{2^n},\tfrac{k}{2^n})\big) \sum_{j=0}^\infty \mu_Q^+\big((a_i^N,b_i^N)\times(c_j^N,d_j^N)\big)
    +\mu_Q^+\big(\big((\tfrac{k-1}{2^n},\tfrac{k}{2^n})  \cap K_N'\big) \times L_N'\big)\\
    &=\sum_{i=0}^\infty \lambda_{(a_i^N,b_i^N)}\big((\tfrac{k-1}{2^n},\tfrac{k}{2^n})\big) \cdot \mu_Q^+\big((a_i^N,b_i^N) \times L_N'\big)
    +\mu_Q^+\big(\big((\tfrac{k-1}{2^n},\tfrac{k}{2^n}) \cap K_N'\big) \times L_N\big)\\
    &\hspace{0.5cm}+\sum_{i=0}^\infty \lambda_{(a_i^N,b_i^N)}\big((\tfrac{k-1}{2^n},\tfrac{k}{2^n})\big) \cdot \mu_Q^+\big((a_i^N,b_i^N) \times L_N\big)
    +\mu_Q^+\big(\big((\tfrac{k-1}{2^n},\tfrac{k}{2^n})  \cap K_N'\big) \times L_N'\big)\\
    &=\sum_{i=0}^\infty \lambda_{(a_i^N,b_i^N)}\big((\tfrac{k-1}{2^n},\tfrac{k}{2^n})\big) \cdot \mu_Q^+\big((a_i^N,b_i^N) \times \II\big)
    +\mu_Q^+\big(\big((\tfrac{k-1}{2^n},\tfrac{k}{2^n}) \cap K_N'\big) \times \II\big)\\
    &=\sum_{i=0}^\infty \frac{\lambda\big((\tfrac{k-1}{2^n},\tfrac{k}{2^n}) \cap (a_i^N,b_i^N)\big)}{\lambda((a_i^N,b_i^N))} \cdot \mu_Q^+\big((a_i^N,b_i^N) \times \II\big)
    +\mu_Q^+\big(\big((\tfrac{k-1}{2^n},\tfrac{k}{2^n}) \cap K_N'\big) \times \II\big).
\end{aligned}
\end{equation}
We now consider two cases.

\medskip{\bf Case 1:} $(\tfrac{k-1}{2^n},\tfrac{k}{2^n}) \subseteq K_N$.
Note that $\tfrac{2k-1}{2^{n+1}} \in (\tfrac{k-1}{2^n},\tfrac{k}{2^n})$ but $\tfrac{2k-1}{2^{n+1}} \notin \bigcup_{m=n+1}^\infty K_{m,N}$ because $\KK_{m,N} \subseteq \mathcal{P}_m$.
Hence, $\tfrac{2k-1}{2^{n+1}} \in \bigcup_{m=0}^n K_{m,N}$ by equation \eqref{eq:K_N}.
This implies that there exists $m$ with $0\le m\le n$ and $A \in \KK_{m,N}$  such that $(\tfrac{k-1}{2^n},\tfrac{k}{2^n})$ intersects $A$. But then $(\tfrac{k-1}{2^n},\tfrac{k}{2^n}) \subseteq A$ because $A \in \mathcal{P}_m$ and $m\le n$.
By equation \eqref{eq:inter} we may assume without loss of generality that $A=(a_0^N,b_0^N)$.
Hence, $(\tfrac{k-1}{2^n},\tfrac{k}{2^n})$ does not intersect any $(a_i^N,b_i^N)$ with $i \ge 1$.
It now follows from inequality \eqref{eq:est_alpha} and from $(\tfrac{k-1}{2^n},\tfrac{k}{2^n}) \cap K_N'=\emptyset$ that
\begin{equation*}
    \mu_N^+\big((\tfrac{k-1}{2^n},\tfrac{k}{2^n}) \times \II\big)
    \le \frac{\lambda((\tfrac{k-1}{2^n},\tfrac{k}{2^n}))}{\lambda((a_0^N,b_0^N))} \cdot \mu_Q^+\big((a_0^N,b_0^N) \times \II\big),
\end{equation*}
or equivalently
\begin{equation*}
    2^n \mu_N^+\big((\tfrac{k-1}{2^n},\tfrac{k}{2^n}) \times \II\big)
    \le \frac{\mu_Q^+(A \times \II)}{\lambda(A)}.
\end{equation*}
Since $A \in \KK_{m,N}$ with $m \le n$ and the set $K_{m,N}$ is disjoint from $\bigcup_{l=0}^{m-1} K_{l,N}$, Lemma~\ref{lem:J_n_N} and equation \eqref{eq:J_family} imply that $A \notin \mathcal{J}_{m,N}$. By condition \eqref{eq:J_cond} we have $\tfrac{\mu_Q^+(A \times \II)}{\lambda(A)} \le N \cdot \mu_Q^+(\II^2)$, hence
\begin{equation}\label{eq:case_1}
2^n \mu_N^+\big((\tfrac{k-1}{2^n},\tfrac{k}{2^n}) \times \II\big) \le N \cdot \mu_Q^+(\II^2).
\end{equation}

\medskip{\bf Case 2:} $(\tfrac{k-1}{2^n},\tfrac{k}{2^n}) \nsubseteq K_N$.
Note that each $(a_i^N,b_i^N)$ is a member of $\bigcup_{n=0}^\infty \mathcal{P}_n$. Hence, we have three options for each $(a_i^N,b_i^N)$, namely, $(i)$ the interval $(\tfrac{k-1}{2^n},\tfrac{k}{2^n})$ does not intersect $(a_i^N,b_i^N)$, $(ii)$ the interval $(\tfrac{k-1}{2^n},\tfrac{k}{2^n})$ is contained in $(a_i^N,b_i^N)$, or $(iii)$ the interval $(\tfrac{k-1}{2^n},\tfrac{k}{2^n})$ contains $(a_i^N,b_i^N)$. By the case assumption and equation \eqref{eq:inter}, option $(ii)$ cannot happen, so the interval $(\tfrac{k-1}{2^n},\tfrac{k}{2^n})$ contains all $(a_i^N,b_i^N)$ that it intersects. In particular, 
\begin{equation*}
\bigcup_{(a_i^N,b_i^N) \subseteq (\frac{k-1}{2^n},\frac{k}{2^n})} (a_i^N,b_i^N) =(\tfrac{k-1}{2^n},\tfrac{k}{2^n}) \cap K_N.
\end{equation*}
Thus, inequality \eqref{eq:est_alpha} implies
\begin{equation}\label{eq:compare}
\begin{aligned}
    \mu_N^+\big((\tfrac{k-1}{2^n},\tfrac{k}{2^n}) \times \II\big)
    &\le \sum_{(a_i^N,b_i^N) \subseteq (\frac{k-1}{2^n},\frac{k}{2^n})} \mu_Q^+\big((a_i^N,b_i^N) \times \II\big)
    +\mu_Q^+\big(\big((\tfrac{k-1}{2^n},\tfrac{k}{2^n}) \cap K_N'\big) \times \II\big)\\
    &=\mu_Q^+\big(\big((\tfrac{k-1}{2^n},\tfrac{k}{2^n}) \cap K_N\big) \times \II\big)
    +\mu_Q^+\big(\big((\tfrac{k-1}{2^n},\tfrac{k}{2^n}) \cap K_N'\big) \times \II\big)\\
    &=\mu_Q^+\big((\tfrac{k-1}{2^n},\tfrac{k}{2^n}) \times \II\big).
\end{aligned}
\end{equation}
Furthermore, $(\tfrac{k-1}{2^n},\tfrac{k}{2^n}) \notin \mathcal{J}_{n,N}$, otherwise Lemma~\ref{lem:J_n_N} and equation \eqref{eq:K_N} would imply that $(\tfrac{k-1}{2^n},\tfrac{k}{2^n}) \subseteq J_{n,N} \subseteq K_N$. By condition \eqref{eq:J_cond} this implies
\begin{equation*}
2^n \mu_Q^+\big((\tfrac{k-1}{2^n},\tfrac{k}{2^n}) \times \II\big)=\frac{\mu_Q^+\big((\tfrac{k-1}{2^n},\tfrac{k}{2^n}) \times \II\big)}{\lambda\big((\tfrac{k-1}{2^n},\tfrac{k}{2^n})\big)} \le N \cdot \mu_Q^+(\II^2).
\end{equation*}
This, together with inequality \eqref{eq:compare}, gives
\begin{equation}\label{eq:case_2}
2^n \mu_N^+\big((\tfrac{k-1}{2^n},\tfrac{k}{2^n}) \times \II\big) \le N \cdot \mu_Q^+(\II^2).
\end{equation}
Combining inequalities \eqref{eq:case_1} and \eqref{eq:case_2} with inequality \eqref{eq:sum_V_Q_N} gives
\begin{equation*}
    2^n \sum_{l=1}^{2^n} V_{Q_N}(R_{kl}^{n})^+ \le N \cdot \mu_Q^+(\II^2).
\end{equation*}
Since $n$ and $k$ were arbitrary, we infer
\begin{equation*}
    \displaystyle\sup_{n \geq 1} \left\{ \max_{k \in [2^n]} 2^n \sum_{l=1}^{2^n} V_{Q_N}(R_{kl}^{n})^+ \right\} \le N \cdot \mu_Q^+(\II^2).
\end{equation*}
By symmetry we also obtain
\begin{equation*}
    \displaystyle\sup_{n \geq 1} \left\{ \max_{l \in [2^n]} 2^n \sum_{k=1}^{2^n} V_{Q_N}(R_{kl}^{n})^+ \right\} \le N \cdot \mu_Q^+(\II^2).
\end{equation*}
Consequently, $\alpha_{Q_N} \le N \cdot \mu_Q^+(\II^2)<\infty$, as required.
\end{proof}

Let us summarize the results of this section that we will need for the proof of Theorem~\ref{thm:main}.

\begin{theorem}\label{thm:main_sequence}
    Let $Q$ be a bivariate quasi-copula that induces a signed measure $\mu_Q$ on $\BB(\II^2)$. Then there exists a sequence of bivariate quasi-copulas $Q_N$, $N \ge 1$, such that
    \begin{enumerate}[$(i)$]
    \item $Q_N=\alpha_N A_N+\beta_N B_N$ for some bivariate copulas $A_N$ and $B_N$ and some real numbers $\alpha_N$ and $\beta_N$, and
    \item the sequence of induced measures $\mu_N=\alpha_N \mu_{A_N}+\beta_N \mu_{B_N}$ converges to $\mu_Q$ in the total variation norm.
    \end{enumerate}
\end{theorem}

\section{Proof of the main theorem}\label{sec:proof}

In this section we continue assuming that $Q$ is a bivariate quasi-copula that induces a signed measure $\mu_Q$. By Theorem~\ref{thm:main_sequence} we obtain a sequence of measures $\mu_N$, induced by quasi-copulas $Q_N$, that converges to $\mu_Q$ in the total variation norm.
We will now convert the sequence of measures $\mu_N$ into a series, which will be manipulated to produce a converging series of multiples of measures induced by copulas.
We will employ a method that was used in \cite{DolKuzSto24} for converting a sequence of quasi-copulas into a function series converging in the supremum norm. We just need to apply the method to a sequence of measures converging in the total variation norm.

First we construct the series
\begin{equation}\label{eq:ser_Q}
    \mu_Q=\mu_1+\sum_{N=2}^\infty (\mu_{N}-\mu_{N-1}).
\end{equation}
Partial sums of this series are measures $\mu_N$, so by Theorem~\ref{thm:main_sequence} the series converges in the total variation norm and its sum is $\mu_Q$. Theorem~\ref{thm:main_sequence} also implies that $Q_N=\alpha_N A_N+\beta_N B_N$ for some bivariate copulas $A_N$ and $B_N$ and some real numbers $\alpha_N$ and $\beta_N$.
Note that $\alpha_N+\beta_N=Q_N(1,1)=1$. If both $\alpha_N$ and $\beta_N$ are positive, $Q_N$ is a convex combination of copulas, so it is a copula itself and we may assume $\alpha_N=1$, $\beta_N=0$, and $A_N=Q_N$. If at least one of $\alpha_N$ and $\beta_N$ is negative, we may assume $\beta_N<0$, and consequently $\alpha_N>1$.
So we have $\alpha_N \ge 1$ and $\beta_N \le 0$ for all $N \ge 1$.
Furthermore,
\begin{align*}
    \mu_{N}-\mu_{N-1} &=\alpha_{N} \mu_{A_{N}}+\beta_{N} \mu_{B_{N}}-\alpha_{N-1} \mu_{A_{N-1}}-\beta_{N-1} \mu_{B_{N-1}}\\
    &=(\alpha_{N}-\beta_{N-1}) \cdot \frac{\alpha_{N} \mu_{A_{N}}-\beta_{N-1} \mu_{B_{N-1}}}{\alpha_{N}-\beta_{N-1}}
    -(\alpha_{N-1}-\beta_{N}) \cdot \frac{\alpha_{N-1} \mu_{A_{N-1}}-\beta_{N} \mu_{B_{N}}}{\alpha_{N-1}-\beta_{N}}
\end{align*}
for all $N \ge 2$, where $\alpha_{N}-\beta_{N-1} \ge 1$ and $\alpha_{N-1}-\beta_{N} \ge 1$. Since $\alpha_N \ge 1$ and $\beta_N \le 0$ for all $N \ge 1$, the functions
\begin{equation}\label{eq:cop_DE}
    D_N=\frac{\alpha_{N} A_{N} - \beta_{N-1} B_{N-1}}{\alpha_{N}-\beta_{N-1}} \qquad\text{and}\qquad E_N=\frac{\alpha_{N-1} A_{N-1} - \beta_{N} B_{N}}{\alpha_{N-1}-\beta_{N}}
\end{equation}
for all $N \ge 2$ are convex combinations of copulas, so they are copulas themselves. Denote also
\begin{equation}\label{eq:num_zx}
    \zeta_N=\alpha_{N}-\beta_{N-1} \ge 1 \qquad\text{and}\qquad \xi_N=-(\alpha_{N-1}-\beta_{N}) \le -1
\end{equation}
for all $N \ge 2$ so that
\begin{align*}
    \mu_{N+1}-\mu_{N} = \zeta_N \mu_{D_N} + \xi_N \mu_{E_N}.
\end{align*}
In addition, let
\begin{equation}\label{eq:cop_num}
    D_1=A_1, \quad E_1=B_1, \quad \zeta_1=\alpha_1, \quad\text{and}\quad \xi_1=\beta_1.
\end{equation}
Then the series in equation \eqref{eq:ser_Q} is expressed as
\begin{equation}\label{eq:ser_DE}
    \mu_Q=\sum_{N=1}^\infty (\zeta_N \mu_{D_N} + \xi_N \mu_{E_N}).
\end{equation}
If we omit the parenthesis in series \eqref{eq:ser_DE}, the resulting series
\begin{equation*}
    \zeta_1 \mu_{D_1} + \xi_1 \mu_{E_1} + \zeta_2 \mu_{D_2} + \xi_2 \mu_{E_2} + \zeta_3 \mu_{D_3} + \xi_3 \mu_{E_3} + \ldots
\end{equation*}
is not convergent in the total variation norm. For example, evaluating it on the set $\II^2$ we obtain the series $\zeta_1 + \xi_1 + \zeta_2 + \xi_2 + \zeta_3 + \xi_3 + \ldots$, which is divergent, in fact, oscillating, since $\zeta_N \ge 1$ and $\xi_N \le -1$.
Before we omit the parenthesis in series \eqref{eq:ser_DE}, we need to split its terms into sums of terms with small enough norm, so that after omitting the parenthesis the "oscillation" will tend to $0$.
For every $N \ge 1$ we choose a positive integer $M_N > |\xi_N|$. We rewrite the series in equation \eqref{eq:ser_DE} as
\begin{equation}\label{eq:ser_DE_split}
    \mu_Q=\sum_{N=1}^\infty \sum_{i=1}^{N M_N} \big(\tfrac{\zeta_N}{N M_N} \mu_{D_N} + \tfrac{\xi_N}{N M_N} \mu_{E_N}\big)
\end{equation}
and remove the parenthesis to obtain the series
\begin{equation}\label{eq:ser_final}
    \begin{aligned}
        &\phantom{+}\underbrace{\tfrac{\zeta_1}{1 M_1} \mu_{D_1} + \tfrac{\xi_1}{1 M_1} \mu_{E_1}
        +\tfrac{\zeta_1}{1 M_1} \mu_{D_1} + \tfrac{\xi_1}{1 M_1} \mu_{E_1}
        +\ldots+\tfrac{\zeta_1}{1 M_1} \mu_{D_1} + \tfrac{\xi_1}{1 M_1} \mu_{E_1}}_{2M_1 \text{ terms}}\\
        &+\underbrace{\tfrac{\zeta_2}{2 M_2} \mu_{D_2} + \tfrac{\xi_2}{2 M_2} \mu_{E_2}
        +\tfrac{\zeta_2}{2 M_2} \mu_{D_2} + \tfrac{\xi_2}{2 M_2} \mu_{E_2}
        +\ldots+\tfrac{\zeta_2}{2 M_2} \mu_{D_2} + \tfrac{\xi_2}{2 M_2} \mu_{E_2}}_{4M_2 \text{ terms}}\\
        &+\underbrace{\tfrac{\zeta_3}{3 M_3} \mu_{D_3} + \tfrac{\xi_3}{3 M_3} \mu_{E_3}
        +\tfrac{\zeta_3}{3 M_3} \mu_{D_3} + \tfrac{\xi_3}{3 M_3} \mu_{E_3}
        +\ldots+\tfrac{\zeta_3}{3 M_3} \mu_{D_3} + \tfrac{\xi_3}{3 M_3} \mu_{E_3}}_{6M_3 \text{ terms}}\\
        &+\ldots
    \end{aligned}
\end{equation}

We now prove the following.

\begin{lemma}\label{lem:series}
    The series of finite signed measures \eqref{eq:ser_final} converges to $\mu_Q$ in the total variation norm.
\end{lemma}

\begin{proof}
The difference between $\mu_Q$ and any partial sum of series \eqref{eq:ser_final}, is of the form
\begin{align*}
    \Delta &=\delta \tfrac{\xi_m}{m M_m} \mu_{E_m} +\sum_{i=k+1}^{m M_m} \big(\tfrac{\zeta_m}{m M_m} \mu_{D_m} + \tfrac{\xi_m}{m M_m} \mu_{E_m}\big)+\sum_{N=m+1}^\infty \sum_{i=1}^{N M_N} \big(\tfrac{\zeta_N}{N M_N} \mu_{D_N} + \tfrac{\xi_N}{N M_N} \mu_{E_N}\big)\\
    &=\delta \tfrac{\xi_m}{m M_m} \mu_{E_m} +\tfrac{m M_m-k}{m M_m} (\zeta_m \mu_{D_m} + \xi_m \mu_{E_m})+\sum_{N=m+1}^\infty (\zeta_N \mu_{D_N} + \xi_N \mu_{E_N})
\end{align*}
for some $m \ge 1$, $k \in [m M_m]$, and $\delta \in \{0,1\}$. We can estimate its total variation norm as follows
\begin{align*}
    \|\Delta\|_{TV} &\le \delta \tfrac{|\xi_m|}{m M_m} \cdot \big\|\mu_{E_m}\big\|_{TV} +\tfrac{m M_m-k}{m M_m} \cdot \big\|\zeta_m \mu_{D_m} + \xi_m \mu_{E_m}\big\|_{TV}\\
    &\hspace{0.5cm}+\Big\|\sum_{N=m+1}^\infty (\zeta_N \mu_{D_N} + \xi_N \mu_{E_N})\Big\|_{TV}.
\end{align*}
Using $M_m > |\xi_m|$ and the fact that $\mu_{E_m}$ is a probability measure, we obtain
\begin{equation*}
    \|\Delta\|_{TV} \le \tfrac{\delta}{m} +\big\|\zeta_m \mu_{D_m} + \xi_m \mu_{E_m}\big\|_{TV}+\Big\|\sum_{N=m+1}^\infty (\zeta_N \mu_{D_N} + \xi_N \mu_{E_N})\Big\|_{TV}.
\end{equation*}
When $m$ tends to infinity, the first term converges to $0$, the second term converges to $0$ because it is the norm of a single term of a converging series \eqref{eq:ser_DE}, and the last term converges to $0$ because it is the norm of the tail of a converging series \eqref{eq:ser_DE}.
This proves that the partial sums of series \eqref{eq:ser_final} converge to $\mu_Q$ in the total variation norm.
\end{proof}

We are now finally ready to prove Theorem~\ref{thm:main}.

\begin{proof}[Proof of Theorem~\ref{thm:main}]
\phantom{1}

$(i) \Longrightarrow (ii)$: Assume that a quasi-copula $Q$ induces a signed measure $\mu_Q$ on $\BB(\II^2)$. Let $C_1,C_2,C_3,\ldots$ be the sequence of copulas
\begin{equation*}
    \underbrace{D_1, E_1, D_1, \ldots, D_1, E_1}_{2M_1 \text{ terms}}
    ,\underbrace{D_2, E_2, D_2, \ldots, D_2, E_2}_{4M_2 \text{ terms}}
    ,\underbrace{D_3, E_3, D_3, \ldots, D_3, E_3}_{6M_3 \text{ terms}}
    ,\ldots
\end{equation*}
defined in equations \eqref{eq:cop_DE} and \eqref{eq:cop_num}, and let $\gamma_1,\gamma_2,\gamma_3,\ldots$ be the sequence of real numbers
\begin{equation*}
    \underbrace{\tfrac{\zeta_1}{1 M_1}, \tfrac{\xi_1}{1 M_1}, \tfrac{\zeta_1}{1 M_1}, \ldots, \tfrac{\zeta_1}{1 M_1}, \tfrac{\xi_1}{1 M_1} }_{2M_1 \text{ terms}}
    ,\underbrace{\tfrac{\zeta_2}{2 M_2}, \tfrac{\xi_2}{2 M_2}, \tfrac{\zeta_2}{2 M_2}, \ldots, \tfrac{\zeta_2}{2 M_2}, \tfrac{\xi_2}{2 M_2}}_{4M_2 \text{ terms}}
    ,\underbrace{\tfrac{\zeta_3}{3 M_3}, \tfrac{\xi_3}{3 M_3}, \tfrac{\zeta_3}{3 M_3}, \ldots, \tfrac{\zeta_3}{3 M_3}, \tfrac{\xi_3}{3 M_3}}_{6M_3 \text{ terms}}
    ,\ldots
\end{equation*}
defined in equations \eqref{eq:num_zx} and \eqref{eq:cop_num}.
By Lemma~\ref{lem:series} the series $\sum_{n=1}^\infty \gamma_n \mu_{C_n}$ converges in the total variation norm to $\mu_Q$, which proves condition $(b)$.
For all $k \geq 1$ and $(x,y) \in \II^2$ we have
\begin{align*}
    \Big|Q(x,y)-\sum_{n=1}^k \gamma_n C_n(x,y)\Big| &= \Big|\Big(\mu_Q - \sum_{n=1}^k \gamma_n \mu_{C_n}\Big)\big([0,x] \times [0,y]\big)\Big|\\
    &\le \Big|\mu_Q - \sum_{n=1}^k \gamma_n \mu_{C_n}\Big|\big([0,x] \times [0,y]\big)\\
    &\le \Big|\mu_Q - \sum_{n=1}^k \gamma_n \mu_{C_n}\Big|(\II^2)
    =\Big\| \mu_Q - \sum_{n=1}^k \gamma_n \mu_{C_n} \Big\|_{TV},
\end{align*}
Hence, the convergence of the series $\sum_{n=1}^\infty \gamma_n \mu_{C_n}$ in the total variation norm implies that the series $\sum_{n=1}^\infty \gamma_n C_n$ converges uniformly to $Q$. This proves condition $(a)$.

$(ii) \Longrightarrow (i)$: By condition $(b)$ the sum of the series $\sum_{n=1}^\infty \gamma_n \mu_{C_n}$ is a finite signed measure on $\BB(\II^2)$. Denote this measure by $\mu$.
As above, this implies that the series $\sum_{n=1}^\infty \gamma_n C_n$ converges uniformly to the function $(x,y) \mapsto \mu\big([0,x] \times [0,y]\big)$. On the other hand, this series converges to $Q$ by condition $(a)$.
Hence, $Q(x,y)=\mu\big([0,x] \times [0,y]\big)$ for all $(x,y) \in \II^2$. 
\end{proof}

We note that if condition $(i)$ of Theorem~\ref{thm:main} holds, then the series in condition $(ii)(b)$ converges to $\mu_Q$. This implies that a measure induced by a bivariate quasi-copula can always be expressed as a converging sum of multiples of measures induced by bivariate copulas, i.e., as what we can call an \emph{infinite linear combination} of measures induced by copulas.

\begin{corollary}\label{cor:measure}
    Any measure $\mu$ induced by some measure-inducing bivariate quasi-copula can be expressed as
    $\mu=\sum_{n=1}^\infty \gamma_n \mu_n,$
    where each $\mu_n$ is a measure induced by some bivariate copula, each $\gamma_n$ is a real number, and the series converges in the total variation norm.
\end{corollary}

As another corollary to Theorem~\ref{thm:main} we obtain the following interesting property.

\begin{theorem}\label{thm:ac}
    Let $Q$ be a bivariate quasi-copula that induces a signed measure $\mu_Q$ on $\BB(\II^2)$. Then
    \begin{equation}\label{eq:H}
        H(x,y)=\frac{|\mu_Q|\big([0,x] \times [0,y]\big)}{|\mu_Q|(\II^2)}
    \end{equation}
    is a joint distribution function of two absolutely continuous random variables $X$ and $Y$ with ranges in $\II^2$.
\end{theorem}

\begin{proof}
    Let $C_n$ and $\gamma_n$ be the sequences from Theorem~\ref{thm:main}, so that $\mu_Q=\sum_{n=1}^\infty \gamma_n \mu_{C_n}$.
    Function $H$ is clearly a distribution function of two random variables $X$ and $Y$ with support in $\II$. 
    The cumulative distribution of $X$ is given by
    \begin{equation*}
        F_X(x)=\frac{|\mu_Q|\big([0,x] \times \II\big)}{|\mu_Q|(\II^2)}.
    \end{equation*}
    Let $\mu_{F_X}$ be the positive measure induced by $F_X$, i.e., $\mu_{F_X}(A)=\frac{|\mu_Q|(A \times \II)}{|\mu_Q|(\II^2)}$ for all $A \in \BB(\II)$.
    Suppose $A \in \BB(\II)$ satisfies $\lambda(A)=0$. Let $P^+$ and $P^-$ be the supports of measures $\mu_Q^+$ and $\mu_Q^-$, respectively.
    Then
    \begin{align*}
        \mu_{F_X}(A) &=\frac{|\mu_Q|(A \times \II)}{|\mu_Q|(\II^2)}=\frac{\mu_Q\big((A \times \II) \cap P^+\big)-\mu_Q\big((A \times \II) \cap P^-\big)}{|\mu_Q|(\II^2)}\\
        &=\frac{1}{|\mu_Q|(\II^2)} \Big(\sum_{n=1}^\infty \gamma_n \mu_{C_n}\big((A \times \II) \cap P^+\big) - \sum_{n=1}^\infty \gamma_n \mu_{C_n}\big((A \times \II) \cap P^-\big)\Big).
    \end{align*}
    Since $\mu_{C_n}$ is a positive measure, $\mu_{C_n}\big((A \times \II) \cap P^+\big) \le \mu_{C_n}(A \times \II)=\lambda(A)=0$ and similarly $\mu_{C_n}\big((A \times \II) \cap P^-\big)=0$ for all $n \ge 1$.
    Hence, $\mu_{F_X}(A)=0$. This shows that measure $\mu_{F_X}$ is absolutely continuous with respect to Lebesgue measure. By Radon-Nikodym theorem there exists a $\BB(\II)$-measurable function $f_X$ such that $F_X(x)=\mu_{F_X}([0,x])=\int_{[0,x]} f_X d\lambda$. Therefore, $X$ is absolutely continuous random variable with density $f_X$. Similarly, $Y$ is also absolutely continuous. 
\end{proof}

\begin{remark}
Theorem~\ref{thm:ac} essentially states that the signed measures $\mu_1(A) = |\mu_Q|(A \times \II)$ and $\mu_2(A) =|\mu_Q|(\II \times A)$ are absolutely continuous with respect to the Lebesgue measure on $\II$.
The conclusion of Theorem~\ref{thm:ac} holds also if we replace measure $|\mu_Q|$ in formula~\eqref{eq:H} by either $\mu_Q^+$ or $\mu_Q^-$.
\end{remark}

\section{Example}\label{sec:example}

We conclude this paper with an example illustrating Theorem~\ref{thm:main} and Corollary~\ref{cor:measure}.
In \cite[Example~13]{DolKuzSto24} the authors construct an example of a quasi-copula $Q$, that induces a finite signed measure, but cannot be written as a linear combination of two copulas. This implies that its induced measure $\mu_Q$ cannot be written as a linear combination of two measures induced by copulas. On the other hand, by Corollary~\ref{cor:measure}, measure $\mu_Q$ can be expressed as an infinite linear combination of measures induced by copulas.
We now find such a representation of $\mu_Q$.

First, we briefly recall the definition of quasi-copula $Q$, for some additional details see \cite[Example~13]{DolKuzSto24}.
For a positive integer $n$ let $Q_n$ be a discrete quasi-copula defined on an equidistant mesh $\{0,\frac{1}{2n+1},\frac{2}{2n+1},\ldots,1\}^2 \subseteq \II^2$, which has mass distributed in a checkerboard pattern (of positive and negative values) within the central diamond-shaped area (with no mass outside this area), so that there are exactly $(n+1)^2$ squares with positive mass $\frac{1}{2n+1}$ and $n^2$ squares with negative mass $-\frac{1}{2n+1}$.
For example, the spread of mass of $Q_3$ is given by the matrix
$$
\left[\begin{smallmatrix}
 0 & 0 & 0 & \frac{1}{7} & 0 & 0 & 0 \\
 0 & 0 & \frac{1}{7} & -\frac{1}{7} & \frac{1}{7} & 0 & 0 \\
 0 & \frac{1}{7} & -\frac{1}{7} & \frac{1}{7} & -\frac{1}{7} & \frac{1}{7} & 0 \\
 \frac{1}{7} & -\frac{1}{7} & \frac{1}{7} & -\frac{1}{7} & \frac{1}{7} & -\frac{1}{7} & \frac{1}{7} \\
 0 & \frac{1}{7} & -\frac{1}{7} & \frac{1}{7} & -\frac{1}{7} & \frac{1}{7} & 0 \\
 0 & 0 & \frac{1}{7} & -\frac{1}{7} & \frac{1}{7} & 0 & 0 \\
 0 & 0 & 0 & \frac{1}{7} & 0 & 0 & 0 \\
\end{smallmatrix}\right]
$$
If we denote by $\widehat{Q}_n$ the bilinear extension of $Q_n$ to $\II^2$, then quasi-copula $Q$ is defined as an ordinal sum of quasi-copulas $\{\widehat{Q}_n\}_{n=1}^\infty$, with respect to the partition $\mathcal{J}=\{J_n\}_{n=1}^\infty$, where $J_n=[a_{n-1},a_n]$ and $a_n=\frac12+\frac14+\ldots+\frac{1}{2^n}=1-\frac{1}{2^n}$.
Since the length of $J_n$ is $\frac{1}{2^n}$, the summand $\widehat{Q}_n$ contributes a total mass of $\frac{1}{2^n}$ to the mass of $Q$.

Denote the product copula by $\Pi$. Note that for every positive integer $n$ the function $\widehat{C}_n=\frac{1}{2n}((2n+1)\Pi-\widehat{Q}_n)$ is clearly grounded and has uniform marginals. In fact, it is a copula since its mass is nowhere negative and its total mass is equal to $1$. We can express
\begin{equation}\label{eq:summand}
    \widehat{Q}_n=(2n+1)\Pi-2nC_n=\Pi+2n(\Pi-\widehat{C}_n).
\end{equation}
For a positive integer $n$ denote by $C_n$ the ordinal sum with respect to partition $\mathcal{J}$, where the $n$-th summand is $\widehat{C}_n$ and all other summands are $\Pi$. In addition, denote by $P$ the ordinal sum with respect to partition $\mathcal{J}$, where all the summands are $\Pi$.

We claim that
\begin{equation}\label{eq:ser_mu_Q}
\mu_Q =\mu_P+2(\mu_P-\mu_{C_1})+4(\mu_P-\mu_{C_2})+6(\mu_P-\mu_{C_3})+\ldots+2n(\mu_P-\mu_{C_n})+\ldots
\end{equation}
Indeed, the series on the right converges in the total variation norm because the support of measure $2n(\mu_P-\mu_{C_n})$ is contained in $J_n \times J_n$, so by equation~\eqref{eq:summand}
\begin{align*}
    \|2n(\mu_P-\mu_{C_n})\|_{TV} &=\tfrac{1}{2^n}\|\mu_{\widehat{Q}_n}-\mu_{\Pi}\|_{TV}= \tfrac{1}{2^n} \cdot 2\cdot (\mu_{\widehat{Q}_n}-\mu_{\Pi})^+(\II^2)\\
    &=\tfrac{1}{2^n} \cdot 2\cdot (n+1)^2\cdot \big(\tfrac{1}{2n+1}-\tfrac{1}{(2n+1)^2}\big)=\tfrac{4n(n+1)^2}{2^n(2n+1)^2},
\end{align*}
and the series $\sum_{n=1}^\infty \tfrac{4n(n+1)^2}{2^n(2n+1)^2}$ is convergent with sum $\approx 2.842$.
In addition, on each $J_n \times J_n$ the sum of the series in \eqref{eq:ser_mu_Q} coincides with $\mu_Q$ by equation~\eqref{eq:summand}, because only the two terms $\mu_P$ and $2n(\mu_P-\mu_{C_n})$ are nonzero there. This implies that the sum of the series in \eqref{eq:ser_mu_Q} is indeed equal to $\mu_Q$.

Using the trick from the proof of Lemma~\ref{lem:series}, we obtain from equation~\eqref{eq:ser_mu_Q} the converging sum expression for $\mu_Q$, namely
\begin{align*}
\mu_Q =\mu_P &+\underbrace{\mu_P-\mu_{C_1}+\mu_P-\mu_{C_1}}_{4 \text{ terms}}\\
&+\underbrace{\tfrac{1}{2}\mu_P-\tfrac{1}{2}\mu_{C_2}+\tfrac{1}{2}\mu_P-\tfrac{1}{2}\mu_{C_2}+\ldots+\tfrac{1}{2}\mu_P-\tfrac{1}{2}\mu_{C_2}}_{16 \text{ terms}}\\
&+\underbrace{\tfrac{1}{3}\mu_P-\tfrac{1}{3}\mu_{C_3}+\tfrac{1}{3}\mu_P-\tfrac{1}{3}\mu_{C_3}+\ldots+\tfrac{1}{3}\mu_P-\tfrac{1}{3}\mu_{C_3}}_{36 \text{ terms}}\\
&+\ldots\\[1em]
&+\underbrace{\tfrac{1}{n}\mu_P-\tfrac{1}{n}\mu_{C_n}+\tfrac{1}{n}\mu_P-\tfrac{1}{n}\mu_{C_n}+\ldots+\tfrac{1}{n}\mu_P-\tfrac{1}{n}\mu_{C_n}}_{4n^2 \text{ terms}}\\
&+\ldots
\end{align*}
Of course, this representation is not unique.

\section*{Acknowledgments}

The author would like to thank Damjana Kokol Bukovšek for useful discussions and comments that helped to improve the presentation of this paper.
The author acknowledges financial support from the ARIS (Slovenian Research and Innovation Agency), research core funding No. P1-0222.

\bibliographystyle{amsplain}
\bibliography{measure_inducing_quasi-copulas}

\end{document}